\documentclass[11pt]{article}

\usepackage{epsfig,epsf,url,enumitem}
\usepackage{etoolbox}
\usepackage{graphicx}
\usepackage[hidelinks]{hyperref}
\usepackage{fullpage}
\usepackage{cite}
\usepackage{xfrac} 
\usepackage{verbatim}
\usepackage[small,bf]{caption}
\usepackage{amsmath,amsfonts}
\usepackage{array}

\usepackage{longtable}

\usepackage{amsbsy}
\usepackage{bbm}

\newcommand{\R}{\mathbb{R}}

\newcommand{\bmtx}{\begin{bmatrix}}
\newcommand{\emtx}{\end{bmatrix}}
\newcommand{\bsmtx}{\left[ \begin{smallmatrix}} 
\newcommand{\esmtx}{\end{smallmatrix} \right]} 
\newcommand{\bmatarray}[1]{\left[\begin{array}{#1}}
\newcommand{\ematarray}{\end{array}\right]}

\newcommand{\norm}[1]{\left\lVert#1\right\rVert}
\newcommand{\bmat}[1]{\begin{bmatrix} #1 \end{bmatrix}}

\newcommand{\tp}{T}
\usepackage{braket,amsfonts}

\usepackage{array}

\usepackage[caption=false]{subfig}
\usepackage{pgfplots}

\usepackage{amsthm}

\newtheorem{theorem}{Theorem}[section]

\newtheorem{lemma}[theorem]{Lemma}
\newtheorem{remark}{Remark}[section]
\newtheorem{proposition}{Proposition}[section]

\usepackage{algorithmic}

\usepackage{graphicx,epstopdf}

\usepackage{amsopn}

\DeclareMathOperator{\diag}{diag}
\DeclareMathOperator*{\minimize}{minimize}
\DeclareMathOperator*{\argmin}{arg min}

\usepackage{xspace}
\usepackage{bold-extra}
\usepackage[most]{tcolorbox}

\colorlet{texcscolor}{blue!50!black}
\colorlet{texemcolor}{red!70!black}
\colorlet{texpreamble}{red!70!black}
\colorlet{codebackground}{black!25!white!25}

\title{Bounds for the Tracking Error and Dynamic Regret of \\Inexact Online Optimization Methods: A General Analysis via Sequential Semidefinite Programs
\thanks{\textbf{Funding}: The work of Usman Syed and Bin Hu was supported
by the NSF Award CAREER under Grant 2048168.}
}

\author{
Usman Syed\thanks{U. Syed is with the Coordinated Science Laboratory and the Department of Electrical and Computer Engineering, University of Illinois at Urbana--Champaign, Email: usyed3@illinois.edu}\and
Emiliano Dall'Anese\thanks{E. Dall'Anese is with the Department of Electrical, Computer and Energy Engineering, University of Colorado Boulder, Email: emiliano.dallanese@colorado.edu}\and
  Bin Hu\thanks{B.~Hu is with the Coordinated Science Laboratory and the Department of Electrical and Computer Engineering, University of Illinois at Urbana--Champaign, Email: binhu7@illinois.edu} 
}

\begin{document}
\maketitle


\begin{abstract}
In this paper, we develop a unified framework for analyzing 
 the tracking error  and dynamic regret of   inexact online optimization methods under a variety of settings. Specifically, we leverage the quadratic constraint approach from control theory to formulate sequential semidefinite programs (SDPs) whose feasible points naturally correspond to  tracking error bounds 
of various inexact online optimization methods including  the inexact online gradient descent (OGD) method, the online gradient descent-ascent method, the online stochastic gradient method, and the inexact proximal online gradient method. We provide exact analytical solutions for our proposed sequential SDPs, and obtain fine-grained tracking error bounds for  the online algorithms studied in this paper. We also provide a simple routine to convert the obtained tracking error bounds into dynamic regret bounds. The main novelty of our analysis is that we derive exact analytical solutions for our proposed sequential SDPs under various inexact oracle assumptions in a unified manner.
\end{abstract}

\section{Introduction}
\label{sec:intro}

The connections between optimization and control theory are deep and fundamental. 
There is a recent trend of research tailoring the quadratic constraint approach from control theory \cite{Megretski1997} for unified analysis and computer-aided design of iterative optimization algorithms~\cite{Lessard2014, nishihara2015,hu17a,fazlyab2017analysis,sundararajan2017robust,hu17b,van2017fastest,aybat2018robust,cyrus2018robust,hu2018dissipativity,dhingra2018proximal,han2019systematic,badithela2019analysis,seidman2019control, xiong2020analytical,mohammadi2020robustness,sundararajan2020analysis,scherer2021convex,hassan2021proximal,hu2021analysis,gannot2021frequency,datar2022robust}. 
In these works, Lyapunov arguments have been combined with quadratic constraints to derive small semidefinite programs (SDPs) whose feasible points correspond to convergence or robustness guarantees of iterative optimization algorithms (see \cite{lessard2022analysis} for a recent survey). 
An advantage of such a SDP-based approach is that the resultant SDP conditions can be typically modified in a modularized manner to address different assumptions on the gradient oracle (e.g. being noisy~\cite{aybat2018robust} or inexact~\cite{Lessard2014,hu2021analysis,gannot2021frequency}).
In the setting where the objective function is stationary and does not change over time, the quadratic constraint approach has led to novel SDP conditions that naturally
cover the complexity analysis of 
momentum-based accelerated methods~\cite{Lessard2014,hu17a,fazlyab2017analysis,van2017fastest,cyrus2018robust,badithela2019analysis,xiong2020analytical,sundararajan2020analysis,scherer2021convex}, proximal/operator-splitting methods~\cite{nishihara2015,dhingra2018proximal,seidman2019control,hassan2021proximal}, distributed gradient-based methods~\cite{sundararajan2017robust,han2019systematic,sundararajan2020analysis,datar2022robust}, and stochastic finite-sum methods~\cite{hu17b,hu2018dissipativity}.
 However, in many practical settings, the objective function is changing over time, and  there is a significant gap between the above quadratic constraint framework and the current theory of time-varying online optimization~\cite{simonetto2020time,hazan2016introduction,shalev2011online,dixit2019online,madden2020bounds}. In this paper, we bridge this gap via developing a sequential SDP approach streamlining the analysis  of inexact online optimization methods.
 

 For illustrative purposes, consider time-varying cost functions $f_t:\R^p \to \R$, where $t\in \{0,1,2\cdots\}$ is the time index. The popular online gradient descent (OGD) method iterates as
\begin{align}
\label{eq:OGD}
x_{t+1}=x_t-\alpha_t \nabla f_t(x_t), 
\end{align}
where at each time step, the decision variable is updated using the gradient of the objective function at time $t$.
Such algorithms take full advantage of the available streaming data and can be used to exploit the time-varying nature of  dynamic environments. 
Noticeably, regret and tracking error have been used as two main performance metrics for online optimization algorithms.  Regret has been extensively studied in the online learning literature and provides a precise characterization of the exploration-exploitation trade-off in decision making tasks including online portfolio selection and web ranking~\cite{hazan2016introduction,shalev2011online}, while tracking error is a popular metric quantifying the tracking capability of online algorithms in applications related to target tracking, signal detection, and path estimation~\cite{simonetto2020time,dixit2019online,madden2020bounds,fazlyab2017prediction}.  Both metrics provide complementary benefits and are vital for understanding the performance of online optimization methods.


In this paper, we are particularly interested in analyzing the tracking error and dynamic regret of inexact online optimization methods. 
Currently, tracking error analysis for online optimization methods are still performed in a case-by-case manner, with the help of deep expert insights~\cite{simonetto2020time,dixit2019online,madden2020bounds,fazlyab2017prediction}. Such analysis may not be generalizable when the underlying assumptions on the oracle are changed. For example, the tracking error analysis of OGD in \cite{madden2020bounds} cannot be directly applied to address the case where the online gradient oracle is subject to some relative error. 
Therefore, it is beneficial to develop more coherent techniques for such analysis.
In this paper, we present a unified SDP-based analysis which covers a large family of inexact online optimization methods in both the deterministic and stochastic settings. Our contributions are summarized as follows.
\begin{enumerate}
    \item We leverage control theory to formulate sequential SDP conditions whose feasible points naturally correspond to upper bounds for tracking error\footnote{Suppose $x_t^*$ is the global minimum for $f_t$, and $\{x_t\}$ is generated by some online algorithm. In this paper, the tracking error is measured as $\norm{x_t-x_t^*}^2$ in the deterministic setting or $\mathbb{E}\norm{x_t-x_t^*}^2$ in the stochastic setting.} of various inexact online optimization methods including inexact OGD, stochastic OGD, and the inexact proximal online gradient descent (IP-OGD) method.  
    \item We present exact analytical solutions to our proposed sequential SDPs, leading to the sharpest tracking error bounds that can be obtained via our proposed SDPs. These new tracking error bounds provide insights for understanding the performance of online optimization methods in the presence of oracle inexactness. In addition, our analysis leads to the first bound on the mean-square tracking error of IP-OGD, complementing the origial $L_1$ error bounds presented in \cite{dixit2019online}.
    \item We provide a simple routine converting our tracking error bounds into dynamic regret bounds. This leads to new dynamic regret bounds for the inexact online algorithms studied in this paper.
\end{enumerate}
\emph{Related work.}
Previously, quadratic constraints have been leveraged to formulate sequential SDPs for analyzing biased stochastic gradient descent (SGD) with inexact oracle \cite{hu2021analysis}. 
 On the conceptual level, our sequential SDP approach can be viewed as an extension of the work in \cite{hu2021analysis} to the online setting. 
 However, a unique novelty of our paper is that we obtain exact analytical solutions for our proposed sequential SDP conditions via rigorous algebraic arguments.
 In contrast, 
  the sequential SDPs  in \cite{hu2021analysis} have not been exactly solved in an analytical manner. Specifically,
  the biased SGD bounds in \cite{hu2021analysis} are obtained via choosing a particular feasible point under the guidance of numerical simulations, and it remains unclear whether the chosen feasible point gives the optimal solution for the original sequential SDPs or not (see \cite[Section 2.2]{hu2021analysis} for a detailed discussion). Therefore, the exact SDP solutions obtained in this paper complement the original analysis in \cite{hu2021analysis} and further strengthen our understanding of the quadratic constraint approach. 
 Another related line of work
that uses SDPs to analyze optimization methods is built upon the formulation of the performance estimation problem
(PEP) \cite{drori2014,taylor2017,taylor2017exact}. 
When the objective function is stationary, the PEP framework can be applied to obtain tight upper and lower bounds for worst-case performance of first-order optimization methods under various settings~\cite{taylor2017,taylor2017exact,ryu2020operator,de2017worst,taylor19a,dragomir2022optimal,kim2016,kim2018generalizing,kim2018another}. 
Currently, the PEP framework has not been adapted to the online setting. 
Given the fundamental connection between the
quadratic constraint approach and the PEP framework \cite{taylor18a}, it is possible to extend the PEP framework for tight tracking error analysis of online optimization methods. This is an interesting future task.

\subsection{Notation}
The $p\times p$ identity matrix and the $p \times p$ zero matrix are denoted as $I_p$ and $0_p$, respectively. The Kronecker product of two matrices $A$ and $B$ is denoted by $A \otimes B$. When a matrix $P$ is negative semidefinite (definite), we will use the notation $P\preceq (\prec) 0$.
A differentiable function $f:\R^p\rightarrow \R$ is $L$-smooth if $\|\nabla f(x)-\nabla f(y)\|\le L \|x-y\|$ for all $x, y\in \R^p$ and
is $m$-strongly convex if
$f(x)\ge f(y)+\nabla f(y)^T (x-y)+\frac{m}{2} \|x-y\|^2$ for all $x,y \in \R^p$.
Let $\mathcal{F}(m,L)$ denote the set of differentiable
functions $f: \R^p \rightarrow \R$  that are $L$-smooth and $m$-strongly convex. 
For any $f \in \mathcal{F}(m,L)$ with $m>0$, there exists a unique $x^*$ satisfying $\nabla f(x^*)=0$, and the following inequality holds for any $x\in \R^p$ \cite[Proposition 5]{Lessard2014}:
\begin{align}
\label{eq:gradient3}
\bmat{ x-x^* \\ \nabla f(x) }^T \bmat{ -2mL I_p & (L+m) I_p \\  (L+m) I_p & -2 I_p} 
\bmat{ x-x^* \\ \nabla f(x) } \ge 0.
\end{align}
However, a function satisfying the above inequality may not be convex. The set of differentiable functions satisfying \eqref{eq:gradient3} with some unique global minimum $x^*$  is denoted as $\mathcal{S}(m,L)$. This class of functions has sector-bounded gradients.  The condition \eqref{eq:gradient3} is closely related to  the regularity condition in the signal processing literature~\cite[Appendix A]{xiong2020analytical}.
Obviously, we have $\mathcal{F}(m,L)\subset \mathcal{S}(m,L)$.

\subsection{Problem setting} In this paper, we consider time-varying objective functions $f_t:\R^p\rightarrow \R$, where we recall that $t$ is the time index. Depending on the algorithm to be analyzed, we will assume either $f_t\in\mathcal{F}(m_t, L_t)$ or $f_t\in \mathcal{S}(m_t, L_t)$. The global minimum of the objective function at time $t$ is denoted as $x_t^*$. We further assume that the sequence $\{\sigma_t\}$ satisfies the following inequality for all $t$:
\begin{align}
\norm{x_{t+1}^*-x_t^*}\le \sigma_t.
\end{align}
Here $\{\sigma_t\}$ can be viewed as prescribed bounds characterizing how fast the global minimum changes over time. 
Finally, for ease of exposition, we adopt the following notation:
\begin{align}\label{Eq:Mu_def}
\begin{split}
  Y_t& := \bmat{2L_t m_t & -\left(L_t +m_t\right)\\ -\left(L_t +m_t\right) &  2}, \,\,\,\,\,\,M_t=\bmat{2m_t & -1 \\-1 & 0},\\
  \mu_t& :=  \left\{
    \begin{array}{ll}
    1-m_t\alpha_t  &  \mbox{if } 0\le \alpha_t \le \frac{2}{m_t+L_t}\\
     \alpha_t L_t-1 &  \mbox{if } \frac{2}{m_t+L_t}\le \alpha_t \le \frac{2}{L_t}
    \end{array}
  \right.
  \end{split}
\end{align}
where $\alpha_t$ is the stepsize for the underlying online optimization method to be analyzed. 

\section{Preliminaries: A sequential SDP framework for algorithm analysis}\label{sec:prelim}
Built upon the work in \cite{hu2021analysis}, we first present a general routine that can be used to formulate sequential SDPs for tracking error analysis of online optimization algorithms. For ease of exposition, we adopt the dissipativity perspective on the quadratic constraint approach from \cite{hu17a}.
First, we introduce the following discrete-time dynamical system in the linear state-space form:
\begin{align}
\label{eq:gen_algo}
\xi_{t+1}=\xi_t +\left( B_t  \otimes I_p \right) w_t,
\end{align}
where $\xi_t\in \R^p$ is the state, $B_t\in \R^{1\times q}$ is some input matrix, and $w_t\in \R^{n_w}$ is the input. To make the dimensions compatible, we have $n_w=pq$. We can
interpret $w_t$ as a driving force, and a key concept from dissipativity theory is the \textit{supply rate} condition which
characterizes the energy change in the state $\xi_t$ 
due to the input $w_t$. In this paper, we consider the supply rate condition in the quadratic constraint form:
 \begin{align}
\label{eq:gen_ineq}
\bmat{ \xi_t \\ w_t }^T (X_t^{(j)}\otimes I_p) \bmat{ \xi_t \\ w_t } \le \Lambda_t^{(j)},
\end{align}
where $X_t^{(j)}  \in \R^{(1+q) \times (1+q)}$ and $\Lambda_t^{(j)} \in \R$ 
 are pre-specified for $j=1, 2,\ldots, J$.  The left side of \eqref{eq:gen_ineq} can be viewed as the energy supplied from the driving force~$w_t$ to the internal state $\xi_t$, and $\Lambda_t^{(j)}$ just provides an upper bound for the supplied energy. 
 Later we will show that many online optimization algorithms can be recast in the form of \eqref{eq:gen_algo} (with $\xi_t$ being the tracking error at $t$) subject to some supply rate condition~\eqref{eq:gen_ineq}. The tracking error analysis reduces to finding tight upper bounds on $\norm{\xi_t}^2$ (in the deterministic setting) or $\mathbb{E}\norm{\xi_t}^2$ (in the stochastic setting).
Our analysis relies on the following standard result from dissipativity theory~\cite{willems72a,hu17a}.
\begin{proposition}\label{prop:main}
Consider \eqref{eq:gen_algo} subject to the supply rate condition~\eqref{eq:gen_ineq}.
 If there exist non-negative scalars $\{\lambda_t^{(j)}\}_{j=1}^J$ and $\rho_t$ such that the following condition~holds
\begin{equation}\label{eq:LMI1}
\bmat{1-\rho_t^2 & B_t \\ B_t^T & B_t^TB_t} \preceq \sum_{j=1}^{J} \lambda_t^{(j)} X_t^{(j)},
\end{equation} 
then the iterations of \eqref{eq:gen_algo} satisfy
$\| \xi_{t+1} \|^2 \le \rho_t^2 \|\xi_t \|^2 + \sum_{j=1}^J \lambda_t^{(j)} \Lambda_t^{(j)}$.
In addition, we can obtain the following bound if \eqref{eq:LMI1} holds for all $t$:
\begin{align}
\label{eq:error}
\norm{\xi_t}^2\le  \left(\prod_{k=0}^{t-1} \rho_k^2 \right) \norm{\xi_0}^2+\sum_{k=0}^{t-1} \left(\left(\prod_{l=k+1}^{t-1}\rho_l^2\right)\left(\sum_{j=1}^J \lambda_k^{(j)} \Lambda_k^{(j)}\right)\right).
\end{align}
\end{proposition}
\begin{proof} For completeness, we briefly sketch the proof.
From \eqref{eq:LMI1}, we directly have
\begin{align}
\label{eq:step1}
 \bmat{ \xi_t \\ w_t }^T \left( \bmat{1-\rho_t^2 & B_t \\ B_t^T & B_t^TB_t} \otimes I_p \right) \bmat{ \xi_t \\ w_t }  \le \sum_{j=1}^{J} \lambda_t^{(j)} \bmat{ \xi_t \\ w_t }^T \left( X_t^{(j)} \otimes I_p \right) \bmat{ \xi_t \\ w_t }.
\end{align}
From \eqref{eq:gen_algo}, it is easy to verify that the left side of \eqref{eq:step1} is equal to $\norm{\xi_{t+1}}^2-\rho_t^2 \norm{\xi_t}^2$. 
Hence we can use the supply rate condition  \eqref{eq:gen_ineq} and the non-negativity of $\{\lambda_t^{(j)}\}_{j=1}^J$ to prove $\| \xi_{t+1} \|^2 \le \rho_t^2 \|\xi_t \|^2 + \sum_{j=1}^J \lambda_t^{(j)} \Lambda_t^{(j)}$. Then we can use induction to show $\norm{\xi_t}^2\le  \left(\prod_{k=0}^{t-1} \rho_k^2 \right) \norm{\xi_0}^2+C_t$
where $\{C_t\}$ is calculated from the recursion $C_{t+1}=\rho_t^2 C_t+\sum_{j=1}^J \lambda_t^{(j)} \Lambda_t^{(j)}$ initialized at $C_0=0$.
It is straightforward to verify that $C_t$ is equal to the second term on the right side of~\eqref{eq:error}. This completes the proof.
\end{proof}

  As we will show later, Proposition \ref{prop:main} can be used to obtain various tracking error bounds for online optimization methods. 
  Fixing $B_t$ and $\{X_t^{(j)}\}$, the condition \eqref{eq:LMI1} is linear in $\rho_t^2$ and $\lambda_t^{(j)}$, and hence becomes a linear matrix inequality (LMI) whose feasibility can be verified using SDP solvers. 
  Let $\mathcal{T}_t \subseteq \R_+^{J+1}$ denote the set of tuples $\left( \rho_t, \lambda_t^{(1)}, \hdots , \lambda_t^{(J)} \right) $ that satisfy the LMI condition \eqref{eq:LMI1}. 
Based on \eqref{eq:error}, any feasible point in the Cartesian product $\prod_{k=0}^{t-1}\mathcal{T}_k$ will lead to an upper bound on $\norm{\xi_t}^2$.  A key issue in computing useful
tracking error bounds  is how to choose $\{\rho_t\}$ and $\{\lambda_t^{(j)}\}$ such that the right side of \eqref{eq:error} is minimized. The best bound that can be obtained for $\norm{\xi_t}^2$ via the LMI \eqref{eq:error} should be the solution of the following minimization problem:
\begin{align}
\label{eq:opt1}
\minimize \limits_{(\rho_k, \lambda_k^{(j)})\in \mathcal{T}_k, 0\le k\le t-1} \quad \left(\prod_{k=0}^{t-1} \rho_k^2 \right) \norm{\xi_0}^2+\sum_{k=0}^{t-1} \left(\left(\prod_{l=k+1}^{t-1}\rho_l^2\right)\left(\sum_{j=1}^J \lambda_k^{(j)} \Lambda_k^{(j)}\right)\right).
\end{align}
By applying the argument in \cite[Section 2.2]{hu2021analysis}, one can show that the above problem can be exactly solved using a sequential SDP approach which is formalized below.

\begin{proposition}
\label{Prop:Greedy_app}
Denote the global optimal solution for \eqref{eq:opt1} as $\hat{U}_t$. Set $\hat{U}_0=U_0=\norm{\xi_0}^2$. Then the sequence~$\{\hat{U}_t\}$ satisfies the following recursion
 \begin{align}
 \label{Eq:Optimal_U}
 \hat{U}_{t+1}:=  \minimize_{(\rho_t, \lambda_t^{(j)}) \in \mathcal{T}_t }\quad \left(\rho_t^2 \hat{U}_t + \sum_{j=1}^{J} \lambda_t^{(j)} \Lambda_t^{(j)}\right).
 \end{align}
\end{proposition}
\begin{proof}
Denote the cost function in \eqref{eq:opt1} as $U_t$.  We can verify that $\{U_t\}$ can be recursively calculated via $U_{t+1} = \rho_t^2 U_t  +\sum_{j=1}^{J} \lambda_t^{(j)}  \Lambda_t^{(j)}$ with the initialization
$U_0=\norm{\xi_0}^2$.
Therefore, we can rewrite \eqref{eq:opt1} as
\begin{align*}
\hat{U}_t:= & \minimize_{(\rho_k, \lambda_k^{(j)})\}_{k=0}^{t-1}}
& & U_t \\
& \text{subject to} & &U_{k+1} =\rho_k^2 U_k + \sum_{j=1}^{J} \lambda_k^{(j)}  \Lambda_k^{(j)},\,\, \left(\rho_k, \lambda_k^{(j)} \right) \in \mathcal{T}_k,\,\, 0\le k \le t -1
\end{align*}
The rest of the proof is almost identical to the proof of Proposition 1 in \cite{hu2021analysis}. We can use a similar dynamical programming argument, and the details are omitted.
\end{proof}
 For every $t$, the subproblem \eqref{Eq:Optimal_U} reduces to a SDP, since the cost is linear in $\rho_t^2$ and $\lambda_t^{(j)}$. 
Hence we can follow a unified analysis routine including the following two steps.
\begin{enumerate}
\item Rewrite the online algorithm to be analyzed in the form of  \eqref{eq:gen_algo} with some supply rate condition \eqref{eq:gen_ineq}. The state $\xi_t$ should be set up as the tracking~error.
\item Apply Proposition \ref{prop:main} to obtain the LMI condition \eqref{eq:LMI1}, and solve the sequential SDP problem \eqref{Eq:Optimal_U} to bound $\norm{\xi_t}^2$. 
\end{enumerate}
For the purpose of deriving closed-form tracking error bounds, we need to obtain the exact analytical solutions of the sequential SDP problem \eqref{Eq:Optimal_U}. This is not an easy task. As a matter of fact, the sequential SDPs in \cite{hu2021analysis} have not been solved analytically. To illustrate this issue, let us consider the exact OGD update rule 
\eqref{eq:OGD}.  We denote $e_t=x_t^*-x_{t+1}^*$, and then the OGD iteration \eqref{eq:OGD} can be rewritten as $x_{t+1}-x_{t+1}^*=x_t-x_t^*-\alpha_t \nabla f_t(x_t)+e_t$. If we set $\xi_t=x_t-x_t^*$, $w_t=\bmat{\nabla f_t(x_t)^\tp & e_t^\tp}^\tp$, and $B_t=\bmat{-\alpha_t & 1}$, then the OGD update rule \eqref{eq:OGD} can be viewed as a special case of \eqref{eq:gen_algo}. 
The supply rate condition depends on the assumptions on $\{f_t\}$. If $f_t\in \mathcal{S}(m_t,L_t)$ and $\norm{x_{t+1}^*-x_t^*}\le \sigma_t$, then we have
\begin{align*}
\bmat{x_t-x_t^*\\ \nabla f_t(x_t)}^\tp\bmat{-2L_t m_t I_p & (L_t+m_t) I_p \\ (L_t+m_t) I_p & -2I_p} \bmat{x_t-x_t^*\\\nabla f_t(x_t)}\ge 0, \,\,\,\, \norm{e_t}^2\le \sigma_t^2.
\end{align*}
It is straightforward to rewrite the above two inequalities in the form of the supply rate condition~\eqref{eq:gen_ineq} (for $j=1,2$) with 
$X_t^{(1)}=\diag( Y_t , 0 )$, $\Lambda_t^{(1)}=0$, $X_t^{(2)}=\diag(0, 0, 1)$, and $\Lambda_t^{(2)}=\sigma_t^2$ (recall that $Y_t$ is defined in \eqref{Eq:Mu_def}). 
Based on Proposition \ref{prop:main}, 
If there exist non-negative scalars $(\rho_t, \lambda_t^{(1)}, \lambda_t^{(2)} )$ such that: 
\begin{align}
\label{eq:LMI_OGD}
\bmat{1-\rho_t^2 & -\alpha_t & 1 \\ -\alpha_t & \alpha_t^2 & -\alpha_t\\ 1 & -\alpha_t & 1-\lambda_t^{(2)}}+\lambda_t^{(1)} \bmat{-2L_t m_t & L_t+m_t & 0 \\ L_t+m_t & -2 & 0 \\ 0 & 0 & 0}\preceq 0,
\end{align}
then the OGD method~\eqref{eq:OGD} satisfies
$\norm{x_{t+1}-x_{t+1}^*}^2 \le \rho_t^2 \norm{x_t-x_t^*}^2+\lambda_t^{(2)}  \sigma_t^2$.
Based on Proposition \ref{Prop:Greedy_app}, the tightest tracking error bound that can be obtained from the LMI \eqref{eq:LMI_OGD} for exact OGD is given by the following sequential SDP problem:
\begin{equation}\label{Eq:OGD_Opt_stg1}
\hat{U}_{t+1}= \minimize_{(\rho_t,\lambda_t^{(1)},\lambda_t^{(2)})\in\mathcal{T}_t}
 \rho_t^2 \hat{U}_t + \lambda_t^{(2)} \sigma_t^2,
\end{equation}
where $\mathcal{T}_t \subseteq \R_+^3$ denotes the set of tuples $\left( \rho_t, \lambda_t^{(1)}, \lambda_t^{(2)} \right) $ that satisfy the condition~\eqref{eq:LMI_OGD}. Getting an exact formula for $\hat{U}_t$ is not easy. Getting an upper bound on $\hat{U}_t$ is less difficult, since one can just choose one particular feasible point without arguing about its optimality. Specifically, we can set $\nu_t=\frac{\sigma_t}{\mu_t \sqrt{\hat{U}_t}}$ and then verify that \eqref{eq:LMI_OGD} is feasible with the following choice of decision variables\footnote{If $0\le \alpha_t\le \frac{2}{L_t+m_t}$, the left side of  \eqref{eq:LMI_OGD} with the substitution of the above parameters becomes
\[
\frac{(1+\nu_t) \left( (m_t+L_t)\alpha_2^2-2\alpha_t \right)}{L_t-m_t}\bmat{m_t^2 & -m_t & 0 \\ -m_t & 1 & 0 \\ 0 & 0 & 0} - \nu_t \bmat{1\\ -\alpha_t \\ -\frac{1}{\nu_t}}\bmat{1 & -\alpha_t & -\frac{1}{\nu_t}}\preceq 0.
\]
The verification for the case with
$ \frac{2}{L_t+m_t} \le \alpha_t \le \frac{2}{L_t}$ is similar. }:
\[
\rho_t^2=\mu_t^2\left(1+ \nu_t\right), \,\,\lambda_t^{(1)}=\frac{\mu_t \alpha_t}{L_t-m_t}\left(1+\nu_t\right),\,\,\lambda_t^{(2)}=1+\frac{1}{\nu_t}.
\]

Consequently, we have $\hat{U}_{t+1}\le (\mu_t\sqrt{\hat{U}_t}+\sigma_t)^2$.
The original sequential SDP analysis in \cite{hu2021analysis} is exactly based on choosing particular feasible points (without arguing about optimality). However, obtaining the tightness is more difficult.
Next, we will develop novel algebraic arguments to derive exact analytical solutions for the sequential SDP problem \eqref{Eq:Optimal_U} formulated for various inexact online optimization methods. 
One of the results that we will develop for the inexact OGD method will actually cover the exact OGD analysis as a special case and states that 
the exact solution for \eqref{Eq:OGD_Opt_stg1} is actually given by $\hat{U}_{t+1}= (\mu_t\sqrt{\hat{U}_t}+\sigma_t)^2$ (see Remark \ref{remark1} in Section~\ref{sec:Inexact_OGD_Abs}).

\section{Tracking error bounds for inexact OGD}
\label{sec:Inexact_OGD}
We start from applying our sequential SDP approach to derive tracking error bounds for the inexact OGD method.

\subsection{Analysis of inexact OGD with additive absolute error} \label{sec:Inexact_OGD_Abs}
The update rule for the inexact OGD method is given~by:

 \begin{align}\label{eq:inexactOGD}
   x_{t+1}=x_t-\alpha_t (\nabla f_t(x_t)+v_t),   
 \end{align}
 
 where $v_t$ is an error term capturing the inexactness in the gradient.
We first~assume  $v_t$ is an absolute error satisfying the bound $\norm{v_t}\le c_t$. We have the following result.
\begin{lemma}
\label{lem:InexactOGD_adv_abs}
Consider the inexact OGD method \eqref{eq:inexactOGD}.
For all $t$, we assume: i)  $f_t\in\mathcal{S}(m_t,L_t)$; ii) $\|x_{t+1}^*-x_t^*\|\le \sigma_t$; iii) $\norm{v_t} \le c_t$ . If there exist non-negative scalars $(\rho_t, \lambda_t^{(1)}, \lambda_t^{(2)},\lambda_t^{(3)} )$ such that the following matrix inequality holds
\begin{align}
\label{eq: LMI_InExOGD_Abs}
\bmat{1-\rho_t^2-2\lambda_t^{(1)}L_t m_t & -\alpha_t+\lambda_t^{(1)} (L_t+m_t) & 1 & -\alpha_t\\ -\alpha_t+\lambda_t^{(1)}(L_t+m_t) & \alpha_t^2-2\lambda_t^{(1)} & -\alpha_t & \alpha_t^2\\ 1 & -\alpha_t & 1-\lambda_t^{(2)} & -\alpha_t \\ -\alpha_t & \alpha_t^2 & -\alpha_t & \alpha_t^2-\lambda_t^{(3)}}\preceq 0,
\end{align}
then the iterates from  \eqref{eq:inexactOGD} satisfy 
$\norm{x_{t+1}-x_{t+1}^*}^2 \le \rho_t^2 \norm{x_t-x_t^*}^2+\lambda_t^{(2)} \sigma_t^2+\lambda_t^{(3)} c_t^2$.
\end{lemma}
\begin{proof}
We can rewrite \eqref{eq:inexactOGD} as
$x_{t+1}-x_{t+1}^*=x_t-x_t^*-\alpha \nabla f_t(x_t) +e_t-\alpha v_t$,
where  $e_t=x_t^*-x_{t+1}^*$. 
If we set $\xi_t=x_t-x_t^*$, $w_t=\bmat{ \nabla f_t(x_t)^\tp &  e_t^\tp &  v_t^\tp}^\tp$, and  $B_t=\bmat{-\alpha_t & 1 & -\alpha_t}$, 
then \eqref{eq:inexactOGD} becomes a special case of \eqref{eq:gen_algo}.
The assumption $f_t\in \mathcal{S}(m_t,L_t)$ ensures  a supply rate condition with 
 $X_t^{(1)} = \diag(Y_t,0,0)$ and $\Lambda_t^{(1)}=0$ (recall that $Y_t$ is defined by \eqref{Eq:Mu_def}). 
The bound $\|x_{t+1}^*-x_t^*\|\le \sigma_t$ leads to the second supply rate condition with $X_t^{(2)} =  \diag( 0, 0, 1, 0 )$ and $\Lambda_t^{(2)}=\sigma_t^2$.
Based on $\norm{v_t} \le c_t$, we have $X_t^{(3)} = \diag( 0,0,0,1)$ and $\Lambda_t^{(3)}=c_t^2$. This leads to the desired LMI.
\end{proof}

Next, we solve the sequential SDP \eqref{Eq:Optimal_U} exactly and obtain the following bound.  

\begin{theorem}\label{th:InexactOGD_abs_param}
 Suppose $\alpha_t\le \frac{2}{L_t}$ for all $t$. Let $\mathcal{T}_t \subseteq \R_+^4$ denote the set of tuples $\left( \rho_t, \lambda_t^{(1)},  \lambda_t^{(2)} , \lambda_t^{(3)} \right) $ that satisfy the condition \eqref{eq: LMI_InExOGD_Abs}. Suppose the assumptions in Lemma~\ref{lem:InexactOGD_adv_abs} hold for all $t$. Then the bound $\hat{U}_t$ defined in \eqref{Eq:Optimal_U} with $J=3$ is given by
\begin{align}
  \label{Eq:U_inexactOGD_abs_full}
  \hat{U}_{t+1}= 
    \left( \mu_t \sqrt{\hat{U}_t}+\sigma_t+ \alpha_t c_t \right)^2.
\end{align}
\end{theorem}
\begin{proof}
 Based on Proposition \ref{Prop:Greedy_app}, we need to solve the following problem:
\begin{equation}\label{Eq:InEx_OGD_Opt_stg1}
\hat{U}_{t+1}= \minimize_{(\rho_t,\lambda_t^{(1)},\lambda_t^{(2)},\lambda_t^{(3)})\in\mathcal{T}_t} \rho_t^2 \hat{U}_t + \lambda_t^{(2)} \sigma_t^2 + \lambda_t^{(3)} c_t^2. 
\end{equation}
Solving \eqref{Eq:InEx_OGD_Opt_stg1} exactly is not easy due to the LMI constraint.  We will transform \eqref{Eq:InEx_OGD_Opt_stg1} to an equivalent problem with simpler constraints.
Set $\lambda_t^{(2)} =1+\frac{1}{\nu_t}$ and $\lambda_t^{(3)} =\alpha_t^2(1+\nu_t)\left( 1+\frac{1}{\zeta_t} \right)$. 
By the Schur complement, \eqref{eq: LMI_InExOGD_Abs} is equivalent to $\zeta_t>0,\,\,\nu_t>0$~and 
\begin{equation*}
\bmat{(1+\nu_t)(1+\zeta_t)-\rho_t^2-2\lambda_t^{(1)} L_t m_t   & \lambda_t^{(1)}\left(L_t+m_t\right)-\alpha_t(1+\nu_t)(1+\zeta_t) \\
 \lambda_t^{(1)}\left(L_t+m_t\right)-\alpha_t(1+\nu_t)(1+\zeta_t)   & \alpha_t^2 (1+\nu_t)(1+\zeta_t)-2 \lambda_t^{(1)} }\preceq 0.
\end{equation*}
Set $\tau_t:=\frac{2 \lambda_t^{(1)}}{(1+\nu_t)(1+\zeta_t)}-\alpha_t^2$. We  have $\lambda_t^{(1)}= \frac{1}{2}\left(\alpha_t^2 +\tau_t \right)(1+\nu_t)(1+\zeta_t)$. Define $h(\tau_t)$~as
\begin{align}\label{eq:h_tau}
h(\tau_t)=1-\alpha_t (L_t+m_t)+\frac{\alpha_t^2(m_t^2+L_t^2)}{2} +\frac{\tau_t\left(L_t-m_t\right)^2}{4} +\frac{\alpha_t^2\left(\alpha_t\left(L_t+m_t\right)-2\right)^2}{4 \tau_t}.   
\end{align}
Then the above matrix inequality condition holds if and only if $\zeta_t >0 $, $\nu_t >0$, $\tau_t \geq 0$, and $(1+\nu_t)(1+\zeta_t)h(\tau_t) \leq \rho_t^2$.
It is straightforward to verify $h(\tau_t)\ge 0$ for all $\tau_t\ge 0$.
 In addition, we can show that the optimal choice of $\tau_t$ minimizing $h(\tau_t)$ has to satisfy $\frac{\tau_t}{4}\left(L_t-m_t\right)^2 =\frac{\alpha_t^2\left(\alpha_t\left(L_t+m_t\right)-2\right)^2}{4 \tau_t}$ and eventually leads to
 $\min_{\tau_t\ge 0} h(\tau_t)=\mu_t^2$, where $\mu_t$ is defined by \eqref{Eq:Mu_def}. Hence we have
 \begin{align*}
\hat{U}_{t+1}&=\min_{\nu_t>0}\,\,\left(\min_{\zeta_t>0}\,\,\left( (1+\nu_t) \left( \min_{\tau_t\ge 0} \left( (1+\zeta_t)h(\tau_t)\hat{U}_t + \alpha_t^2 c_t^2 \left( 1+\frac{1}{\zeta_t}\right)  +\frac{\sigma_t^2}{\nu_t} \right)\right)\right) \right) \\
    &=\min_{\nu_t>0}\,\,\left(\min_{\zeta_t>0}\,\, \left( (1+\nu_t) \left( (1+\zeta_t)\left(\min_{\tau_t\ge 0}h(\tau_t)\right)\hat{U}_t +  \alpha_t^2 c_t^2 \left( 1+\frac{1}{\zeta_t}\right)  + \frac{\sigma_t^2}{\nu_t}  \right)\right) \right) \\
    &=\min_{\nu_t>0}\,\,\left(\min_{\zeta_t>0}\,\, \left(  (1+\nu_t) \left( (1+\zeta_t)\mu_t^2\hat{U}_t + \left( 1+\frac{1}{\zeta_t}\right) \alpha_t^2 c_t^2  +\frac{\sigma_t^2}{\nu_t}  \right)\right) \right) 
\end{align*}
Now we can solve \eqref{Eq:InEx_OGD_Opt_stg1} as 
\begin{align*}
\hat{U}_{t+1}        =\min_{\nu_t>0} \left((1+\nu_t)\left(\mu_t \sqrt{\hat{U}_t}+\alpha_t c_t\right)^2+\alpha_t^2 \left(1+\frac{1}{\nu_t}\right)\sigma_t^2\right)
    =\left(\mu_t \sqrt{\hat{U}_t}+\sigma_t+\alpha_t c_t\right)^2.
\end{align*}
This completes our proof.
\end{proof}

We can see that \eqref{Eq:U_inexactOGD_abs_full} gives a simple analytical formula  describing the recursion of~$\{\hat{U}_t\}$, and $\hat{U}_t$ can be easily rewritten as a function of $\hat{U}_0$ and $\{\mu_k,\sigma_k\}_{k=0}^{t-1}$ given the fact $\sqrt{\hat{U}_{t+1}}=\mu_t \sqrt{\hat{U}_t}+\sigma_t+\alpha_t c_t$.
In Section~\ref{sec:dynamicregret}, we will show how to convert this recursive tracking error bound into a dynamic regret bound.
It is also easy to verify that the choice of $\alpha_t$ that minimizes the iterative bound in \eqref{Eq:U_inexactOGD_abs_full} is given by $\hat{\alpha}_t=\frac{2}{L_t+m_t}$. The recursive bound in \eqref{Eq:U_inexactOGD_abs_full} can be further simplified if we take $L_t=L$, $m_t=m$, and $\alpha_t=\alpha$ for all $t$. 
In this case, we have $\mu_t=\mu$ $\forall t$
with $\mu$ given~by
\begin{align*}
  \mu:=  \left\{
    \begin{array}{ll}
    1-m\alpha  & \quad  \mbox{if } 0\le \alpha \le \frac{2}{m+L}\\
     \alpha L-1 & \quad  \mbox{if } \frac{2}{m+L}\le \alpha \le \frac{2}{L}
    \end{array}
  \right.,
\end{align*}
Then \eqref{Eq:U_inexactOGD_abs_full} can be simplified as $\sqrt{\hat{U}_{t+1}}=\mu\sqrt{\hat{U}_t}+\sigma+\alpha c$, which can be equivalently rewritten as  $\sqrt{\hat{U}_t}= 
    \mu^t\left( \sqrt{\hat{U}_0} - \frac{\sigma+ \alpha c }{1-\mu} \right) + \frac{\sigma+ \alpha c }{1-\mu}$.
Clearly, $\sqrt{\hat{U}_t}$ converges linearly to the steady state value $(\sigma+ \alpha c )/(1-\mu)$ with a rate specified by $\mu$.

\begin{remark}[Connections with exact OGD results]\label{remark1}
If $c_t=0$, we can modify the above argument to show that the exact analytical solution to \eqref{Eq:OGD_Opt_stg1} is given by $\hat{U}_{t+1}= (\mu_t\sqrt{\hat{U}_t}+\sigma_t)^2$. In contrast, the previous feasible point argument in Section~\ref{sec:prelim} only proves $\hat{U}_{t+1}\le (\mu_t\sqrt{\hat{U}_t}+\sigma_t)^2$.
If $m_t=m$, $L_t=L$, and $\alpha_t=\alpha$ for all $t$, the formula for $\hat{U}_t$ becomes
$\sqrt{\hat{U}_t}=\mu^t \left(\sqrt{\hat{U}_0}-\frac{\sigma}{1-\mu}\right)+\frac{\sigma}{1-\mu}$.
Clearly, $\sqrt{\hat{U}_t}$ converges to its steady-state value (which is $\sigma/(1-\mu)$) at a linear rate given by~$\mu$. Since we have $\norm{x_t-x_t^*}\le \sqrt{U_t}$, we can see that our analysis recovers the previous exact OGD bound in \cite[Theorem 3.1]{madden2020bounds}. If $\sigma \equiv 0$, we can also recover the convergence bound for the case of static optimization, i.e. $\sqrt{\hat{U}_t} \leq \mu^t \sqrt{\hat{U}_0}$. This provides a good sanity check for our unified analysis.
\end{remark}

\subsection{Analysis of inexact OGD with additive relative error}\label{sec:OGD_relative_err}
Next, we consider the inexact OGD method \eqref{eq:inexactOGD} with a more complicated error model.
We assume that the bias term $v_t$ satisfies $\norm{v_t}\le \delta_t \norm{\nabla f_t(x_t)}$ for some known constant $\delta_t$. 
Such a relative error bound has been extensively studied in the stationary cost setting~\cite{Lessard2014,de2017worst,gannot2021frequency,hu2021analysis}. 
As commented in \cite{de2017worst}, this assumption means  $|\sin (\theta)| \le \delta_t$ with $\theta$ being the angle between $\nabla f_t(x_t)$ and the true update direction $(\nabla f_t(x_t)+v_t)$. 
Now we discuss the impacts of such a relative error term on the tracking capability of OGD.
We have the following LMI condition.

\begin{lemma}
\label{lem:InexactOGD_adv_rel}
Consider the inexact OGD method \eqref{eq:inexactOGD}.
For all $t$, we assume: i)  $f_t \in \mathcal{S}\left(m_t,L_t\right)$; ii) $\|x_{t+1}^*-x_t^*\|\le \sigma_t$; iii) $\norm{v_t}\le \delta_t \norm{\nabla f_t(x_t)} $ for $\delta_t >0$. If there exist non-negative scalars $(\rho_t, \lambda_t^{(1)}, \lambda_t^{(2)},\lambda_t^{(3)} )$ such that the following matrix inequality holds
\begin{align}
\label{eq:LMI_InexOGD_rel}
\bmat{1-\rho_t^2-2\lambda_t^{(1)}L_t m_t & -\alpha_t+\lambda_t^{(1)} (L_t+m_t) & 1 & -\alpha_t\\ -\alpha_t+\lambda_t^{(1)}(L_t+m_t) & \alpha_t^2-2\lambda_t^{(1)}+\delta_t^2 \lambda_t^{(3)} & -\alpha_t & \alpha_t^2\\ 1 & -\alpha_t & 1-\lambda_t^{(2)} & -\alpha_t \\ -\alpha_t & \alpha_t^2 & -\alpha_t & \alpha_t^2-\lambda_t^{(3)}}\preceq 0,
\end{align}
then  the inexact OGD update \eqref{eq:inexactOGD} satisfies $\norm{x_{t+1}-x_{t+1}^*}^2 \le \rho_t^2 \norm{x_t-x_t^*}^2+\lambda_t^{(2)} \sigma_t^2$.
\end{lemma}
\begin{proof}
We can adopt $(X_t^{(1)}, \Lambda_t^{(1)}, X_t^{(2)}, \Lambda_t^{(2)})$ from the proof of Lemma~\ref{lem:InexactOGD_adv_abs}. The inequality $\norm{v_t}\le \delta_t\norm{\nabla f_t(x_t)}$ can be rewritten as a third supply rate condition with $X_t^{(3)}= \diag(0,-\delta_t^2,0,1)$ and $\Lambda_t^{(3)}=0$. This leads to the desired conclusion.
\end{proof}

It is important to notice that the LMI in Lemma \ref{lem:InexactOGD_adv_rel} is only feasible when some reasonable bound on $\delta_t$ is posed.  We will provide such bounds in our next result. We can see that  \eqref{eq:LMI_InexOGD_rel} and \eqref{eq: LMI_InExOGD_Abs} are quite similar. The only difference is that there is an extra term $\delta_t^2 \lambda_t^{(3)}$ in the $(2,2)$-th entry of the LMI \eqref{eq:LMI_InexOGD_rel}. However, such a small change will cause the analytical solution of the resultant SDP problem  \eqref{Eq:Optimal_U} to become much more complicated. We formalize this result as below.

\begin{theorem}\label{th:InexactOGD_rel_param}
Suppose $0\le \alpha_t\le \frac{2}{(1+\delta_t)L_t}$ and  $\delta_t \in [0, \frac{2 m_t}{L_t+m_t} )$ for all $t$. Let $\mathcal{T}_t \subseteq \R_+^4$ denote the set of tuples $\left( \rho_t, \lambda_t^{(1)}, \lambda_t^{(2)},\lambda_t^{(3)} \right) $ that satisfy the condition \eqref{eq:LMI_InexOGD_rel}. Suppose the assumptions in Lemma~\ref{lem:InexactOGD_adv_rel} hold for all $t$. Then the recursive bound $\hat{U}_t$ defined in \eqref{Eq:Optimal_U} with $J=3$ is given by $\hat{U}_{t+1} =       \left(\hat{\rho}_t\sqrt{\hat{U}_t}+\sigma_t\right)^2$,
with $\hat{\rho}_t$ being defined~as
\begin{align}
\label{eq:rho_delta}
  \hat{\rho}_t :=  \left\{
    \begin{array}{ll}
        1-\alpha_t m_t(1-\delta_t)  & \mbox{if } 0\le \alpha_t \le \alpha_{t_-} \\
       \left(1-\frac{2\alpha_t L_t m_t}{L_t+m_t}+\frac{\alpha_t\delta_t^2(L_t+m_t-2\alpha_t L_tm_t)}{2-\alpha_t(L_t+m_t)}\right)^\frac{1}{2} & \mbox{if } \alpha_{t_-} \le \alpha_t \le \alpha_{t_+}\\
      (1+\delta_t)\alpha_t L_t-1 & \mbox{if } \alpha_{t_+} \le \alpha_t \le \frac{2}{(1+\delta_t)L_t}
    \end{array}
  \right.
\end{align} 
for $\alpha_{t_-}:=\frac{1}{1-\delta_t}(\frac{2}{L_t+m_t}-\frac{\delta_t}{m_t})$ and $\alpha_{t_+}:=\frac{1}{1+\delta_t}(\frac{2}{L_t+m_t}+\frac{\delta_t}{L_t})$.
\end{theorem}
\begin{proof}
Based on Proposition \ref{Prop:Greedy_app}, we need to solve the following problem:
\begin{equation}\label{Eq:InEx_OGD_rel_stg1}
\hat{U}_{t+1}= \minimize_{(\rho_t,\lambda_t^{(1)},\lambda_t^{(2)},\lambda_t^{(3)}  )\in\mathcal{T}_t} \rho_t^2 \hat{U}_t + \lambda_t^{(2)} \sigma_t^2.
\end{equation}
In order to solve \eqref{Eq:InEx_OGD_rel_stg1} analytically, we will apply a change of variables again. Set $\lambda_t^{(2)} = 1+\frac{1}{\nu_t}$ and $\lambda_t^{(3)} =\alpha_t^2\left( 1+\nu_t \right)\left(1+\frac{1}{\zeta_t}\right)$. Then we can apply the Schur complement to convert \eqref{eq:LMI_InexOGD_rel} to the following equivalent form:
\[
 \bmat{ \psi\left(\zeta_t,\nu_t\right)-\rho_t^2 -2 L_t m_t \lambda_t^{(1)}  & \lambda_t^{(1)} \left(L_t+m_t\right) -\alpha_t \psi\left(\zeta_t,\nu_t\right)  \\ \lambda_t^{(1)} \left(L_t+m_t\right)-\alpha_t \psi\left(\zeta_t,\nu_t\right)  & \alpha_t^2 \psi\left(\zeta_t,\nu_t\right) \left( 1+\frac{\delta_t^2}{\zeta_t} \right) -2\lambda_t^{(1)}  } \preceq 0, \,\, \zeta_t>0 , \,\, \nu_t>0,
\]
where $\psi(\zeta_t,\nu_t):=(1+\zeta_t)(1+\nu_t)$.
Then we set $\tau_t:=2 \lambda_t^{(1)}-\alpha_t^2\psi(\zeta_t,\nu_t)\left(1+\frac{\delta_t^2}{\zeta_t}\right)$. Consequently, we have $\lambda_t^{(1)} =\frac{1}{2} \left( \tau_t +\alpha_t^2 \psi(\zeta_t,\nu_t)\left(1+\frac{\delta_t^2}{\zeta_t}\right) \right)$.
By applying the Schur complement and setting $\chi_t:= \alpha_t \left( 1+\frac{\delta_t^2}{\zeta_t} \right)$, we can show that the above LMI holds if and only if $\zeta_t>0$, $\nu_t>0$, $\tau_t \geq 0$ and $h(\zeta_t,\nu_t,\tau_t)\le \rho_t^2$, where $h(\zeta_t,\nu_t,\tau_t)$ is defined as
\begin{align*}
        h\left( \zeta_t,\nu_t,\tau_t \right) = - \psi(\zeta_t,\nu_t)  & \left(\alpha_t L_t m_t\chi_t  -1 \right) +\frac{\alpha_t^2 \psi(\zeta_t,\nu_t)^2}{4 \tau_t} \left( 2-\chi_t  \left(L_t + m_t\right)  \right)^2 +\frac{\tau_t}{4} \\ 
        & \left(L_t-m_t\right)^2 -\frac{\alpha_t \psi(\zeta_t,\nu_t)}{2} \left(L_t+m_t\right)\left( 2-\chi_t \left(L_t+m_t\right)\right).
\end{align*}
Therefore, \eqref{Eq:InEx_OGD_rel_stg1} is equivalent to the following optimization:
\begin{align}\label{eq:opt_U_rela}
    \hat{U}_{t+1}&=\min_{\nu_t>0}\,\,\left(\min_{\zeta_t>0}\,\,\left(\min_{\tau_t\geq 0}\left( \min_{h\left( \zeta_t,\nu_t,\tau_t \right) \le \rho_t^2} \left(\rho_t^2\hat{U}_t + \left( 1+\frac{1}{\nu_t}\right)\sigma_t^2 \right)\right)\right) \right).
\end{align}
Notice $1-\alpha_t (m_t+L_t)+\frac{\alpha_t^2 (m_t^2+L_t^2)}{2}\ge 0$, and hence we have $h(\zeta_t, \nu_t, \tau_t)\ge 0$ for all $\nu_t>0$, $\zeta_t>0$, and $\tau_t\ge 0$. Hence we can simplify \eqref{eq:opt_U_rela} as
\begin{align*}
    \hat{U}_{t+1}&=\min_{\nu_t>0}\,\,\left(\min_{\zeta_t>0}\,\,\left(\min_{\tau_t\geq 0}\left( h(\zeta_t,\nu_t,\tau_t)\hat{U}_t + \left( 1+\frac{1}{\nu_t}\right)\sigma_t^2 \right)\right)\right)\\
    &=\min_{\nu_t>0}\,\,\left(\min_{\zeta_t>0}\,\,\left( \hat{h}(\zeta_t,\nu_t)\hat{U}_t + \left( 1+\frac{1}{\nu_t}\right)\sigma_t^2 \right)\right)\\
    &=\min_{\nu_t>0}\,\,\left(\left(\min_{\zeta_t>0}\hat{h}(\zeta_t,\nu_t)\right)\hat{U}_t + \left( 1+\frac{1}{\nu_t}\right)\sigma_t^2 \right),
\end{align*}
where $\hat{h}(\zeta_t,\nu_t):=\min_{\tau_t\ge 0} h(\zeta_t,\nu_t,\tau_t)$ yields the following explicit formula:
\begin{align}\label{eq:h_hat1}
    \begin{split}
        \hat{h}\left( \zeta_t,\nu_t \right) := \psi(\zeta_t,\nu_t) & \left(1-\alpha_t L_t m_t\chi_t \right) + \frac{\alpha_t \psi(\zeta_t,\nu_t) \left(L_t-m_t\right) }{2}  \Biggl| \alpha_t \left( L_t +m_t \right)  \\ 
        &\left(1 +\frac{\delta_t^2}{\zeta_t} \right)  - 2 \Biggr| -\frac{\alpha_t \psi(\zeta_t,\nu_t) \left(L_t+m_t\right)}{2} \left( 2-\chi_t \left(L_t+m_t\right)\right).
    \end{split}
\end{align}
Clearly, the above optimal $\hat{h}$ is achieved at
$\tau_t  =
\frac{\alpha_t \psi(\zeta_t,\nu_t)  }{L_t - m_t}\left| \alpha_t \left( L_t +m_t \right) \left(1 +\frac{\delta_t^2}{\zeta_t} \right) - 2 \right|$.
The rest of the proof relies on 
the following key observation: 
\begin{align}\label{eq:hath}
\min_{\zeta_t>0} \hat{h}(\zeta_t,\nu_t)=(1+\nu_t)\hat{\rho}_t^2,
\end{align}
where $\hat{\rho}_t$ is defined by \eqref{eq:rho_delta}. Once \eqref{eq:hath} is shown, it is trivial to verify
\[
\hat{U}_{t+1}=\min_{\nu_t>0} \left((1+\nu_t)\hat{\rho}_t^2 \hat{U}_t+\left(1+\frac{1}{\nu_t}\right)\sigma_t^2\right)=\left(\hat{\rho}_t \sqrt{\hat{U}_t}+\sigma_t\right)^2,
\]
which directly leads to the desired conclusion. 

To verify \eqref{eq:hath}, we will use \eqref{eq:h_hat1} to derive 
the analytical form of $\min_{\zeta_t>0} \hat{h}(\zeta_t,\nu_t)$, which heavily relies on the sign of
 $\bar{\zeta_t}:=\frac{\delta_t^2 \alpha_t \left( L_t +m_t \right)}{2-\alpha_t \left( L_t +m_t \right)}$.
If $\alpha_t \ge \frac{2}{L_t +m_t}$, then we have $\bar{\zeta_t}\le 0$. 
 In this case, we can apply \eqref{eq:h_hat1} to show
$\hat{h}\left( \zeta_t,\nu_t \right)=(1+\zeta_t)(1+\nu_t)\left( \left(\alpha_t L_t-1\right)^2+\frac{\alpha_t^2 \delta_t^2 L_t^2}{\zeta_t} \right)$ for all $\zeta_t>0$. Then we can minimize $\hat{h}$
via choosing $\zeta_t=\frac{\alpha_t \delta_t L_t}{|\alpha_t L_t-1|}$. Consequently, \eqref{eq:hath} holds for $\alpha\ge \frac{2}{m_t+L_t}$. Notice we require $\hat{\rho}_t<1$ such that the iterative bound makes sense. This poses an upper bound on the stepsize $\alpha_t$. Specifically, we need $\alpha_t\le \frac{2}{(1+\delta_t) L_t}$. 

If $\alpha_t< \frac{2}{m+L}$, we have
$\bar{\zeta_t} >0$. Then \eqref{eq:h_hat1} leads to
\begin{align}\label{eq:Inexact_OGD_rel_h} \hat{h} (\zeta_t, \nu_t) =\begin{cases}
(1+\zeta_t)(1+\nu_t)\left( \left(\alpha_t L_t-1\right)^2+\frac{\alpha_t^2 \delta_t^2 L_t^2}{\zeta_t} \right) \,\ &\text{if} \,\,\,  0 < \zeta_t \leq \bar{\zeta_t} \\
(1+\zeta_t)(1+\nu_t) \left( \left(1-\alpha_t m_t\right)^2+\frac{\alpha_t^2 \delta_t^2 m_t^2}{\zeta_t} \right) \,\ & \text{if} \,\,\,  \zeta_t>\bar{\zeta_t}  
\end{cases}
\end{align}
To obtain the analytical form of $\min_{\zeta_t>0} \hat{h}(\zeta_t,\nu_t)$, we define
 $\zeta_{tL}:=\frac{\alpha_t \delta_t L_t}{\left| \alpha_t L_t -1 \right| }$ and $\zeta_{tm}:=\frac{\alpha_t \delta_t m_t}{1- m_t \alpha_t} $. 
 Due to the fact $\alpha_t<\frac{2}{m_t+L_t}$, we can show $\zeta_{tm}\le \zeta_{tL}$. The analytical form of $\min_{\zeta_t>0} \hat{h}(\zeta_t,\nu_t)$ depends on the ordering relation between $\bar{\zeta}_t$ and $(\zeta_{tm},\zeta_{tL})$ as follows. 
 \begin{itemize}
     \item If $\bar{\zeta}_t\le \zeta_{tm}\le \zeta_{tL}$, we can fix $\nu_t$ and show $\hat{h}$ is monotonically decreasing for $0<\zeta_t\le \bar{\zeta}_t$. Since $\zeta_{tm}\ge \bar{\zeta}_t$, we know $\hat{h}$ is monotonically decreasing for $\zeta_t\le \zeta_{tm}$ and monotonically increasing for $\zeta_t\ge \zeta_{tm}$. Therefore, $\hat{h}$ attains the minimum value at $\zeta_t=\zeta_{tm}$. The relationship $\bar{\zeta}_t\le \zeta_{tm}$ can be equivalently transformed into the stepsize bound $0\le \alpha_t \le \alpha_{t_-} $. 
     
     \item If $ \zeta_{tm} \leq \bar{\zeta_t}  \leq \zeta_{tL}$, then $\hat{h}$ is monotonically decreasing for $0<\zeta_t \le \bar{\zeta}_t$ and monotonically increasing for $\zeta_t\ge \bar{\zeta}_t$. Hence $\hat{h}$ attains the global minimum at $\zeta_t=\bar{\zeta}_t$. 
     The relationship $ \zeta_{tm} \leq \bar{\zeta_t}\leq \zeta_{tL}$ can be transformed into the stepsize bound $\alpha_{t_-} \leq \alpha_t \leq \alpha_{t_+}$.
     
     \item If $ \zeta_{tm} < \zeta_{tL} \leq \bar{\zeta_t}$, we can apply a similar argument to show that $\hat{h}$ is monotonically decreasing for $\zeta_t \le \zeta_{tL}$
    and monotonically increasing for $\zeta_t\ge \zeta_{tL}$.  
     Therefore $\hat{h}$ attains its minimum at $\zeta_t=\zeta_{tL}$, and the bound $ \zeta_{tm} < \zeta_{tL} \leq \bar{\zeta_t}$ is equivalent to the stepzise bound $\alpha_t \ge \alpha_{t_+} $.
  \end{itemize}
To summarize, for any $\nu_t>0$, the function $\hat{h}(\zeta_t,\nu_t)$ is minimized via choosing $\zeta_t =\frac{\alpha_t \delta_t m_t}{1- m_t \alpha_t}$ for $0\le \alpha_t \le \alpha_{t_-}$, $\zeta_t = \frac{\delta_t^2 \alpha_t \left( L_t +m_t \right)}{2-\alpha_t \left( L_t +m_t \right)}$ for $\alpha_{t_-} \leq \alpha_t \leq \alpha_{t_+}$, and $\zeta_t =\frac{\alpha_t \delta_t L_t}{ \alpha_t L_t -1}$ for $\alpha_{t_+} \le \alpha_t \le \frac{2}{(1+\delta_t)L_t}$.
Hence \eqref{eq:hath} holds as desired, and this completes our proof.
\end{proof}

The above theorem characterizes the dependence of the tracking error on $\delta_t$. Obviously, increasing $\delta_t$ leads to slower convergence speed and larger steady-state error. If $\delta_t=0$ for all $t$, our bound reduces to the exact OGD bound, i.e. $\sqrt{\hat{U}_{t+1}}=\mu_t \sqrt{\hat{U}_t}+\sigma_t$. If $\sigma_t=0$ (i.e. $x_{t+1}^*=x_t^*$) for all $t$ and $(m_t,L_t,\delta_t)$ do not change over $t$, our bound reduces to \cite[Propositions 1.3 \& 1.4]{gannot2021frequency}, which give the convergence rate of the inexact gradient method subject to a relative error in the stationary cost setting.
\begin{remark}
The recursive bound $\hat{U}_{t+1} =(\hat{\rho}_t\sqrt{\hat{U}_t}+\sigma_t)^2$ can be further simplified if we take $L_t=L$, $m_t=m$, $\alpha_t=\alpha$ and $\delta \geq \delta_t$ for all $t$ . In this case, we have $\hat{\rho}_t=\hat{\rho}$.
The simplified convergence bound is given by $\sqrt{\hat{U}_{t+1}}=\hat{\rho}\sqrt{\hat{U}_t}+\sigma_t$ which leads us to the asymptotic bound $\sqrt{\hat{U}_t}= \hat{\rho}^t\left( \sqrt{\hat{U}_0} - \frac{\sigma}{1-\hat{\rho}} \right) + \frac{\sigma }{1-\hat{\rho}}$. Clearly, $\sqrt{\hat{U}_t}$ converges linearly to the steady state value $\sigma/(1-\hat{\rho})$ with the convergence rate quantified by $\hat{\rho}$.
\end{remark}

\subsection{Analysis in the inexact variational inequality setting}

Now we consider a different generalized setting where
the inexact OGD method is extended to  the variational inequality (VI) problem. 
At every time step~$t$, the VI problem considers 
finding the point $x_t^*$ such that $F_t(x_t^*)^\tp (x-x_t^*)\ge 0\,\,\forall x\in \mathcal{X}$, where $F_t:\R^p \rightarrow \R^p$ is a vector field, and $\mathcal{X}$ is a prescribed convex set. For simplicity, we consider $\mathcal{X}=\R^p$, and hence $x_t^*$ satisfies
$F_t(x_t^*)=0$. The following algorithm can be viewed as the extension of the inexact OGD method to the VI setting:
\begin{align}\label{eq:VI_OGD}
    x_{t+1}=x_t-\alpha_t (F_t(x_t)+v_t).
\end{align}
We assume $v_t$ is a bias term satisfying a general error bound $\norm{v_t}^2 \le \delta_t^2 \norm{F_t(x_t)}^2 + c_t^2 $ for some $\delta_t$ and $c_t$. 
If $F_t$ happens to be the gradient of some function $f_t$, then the above algorithm reduces to the inexact OGD method. However, the above VI setting is more general, and can cover the minimax game problem where \eqref{eq:VI_OGD} actually becomes the online version of the gradient descent-ascent method \cite{zhang2021unified}. 

We are interested in analyzing how closely the iterates $\{x_t\}$ generated by \eqref{eq:VI_OGD} can track the sequence $\{x_t^*\}$. We make the standard assumptions \cite{zhang2021unified} that $F_t$ is strongly monotone\footnote{The mapping $F_t$ is said to be $m_t$-strongly-monotone if $(F_t(x)-F_t(\tilde{x}))^\tp (x-\tilde{x})\ge m_t\norm{x-\tilde{x}}^2$ for all $x,\tilde{x}\in \R^p$ for a given $m_t>0$. } and Lipschitz. These assumptions lead to new supply rate conditions which are used to formulate the following LMI condition.
\begin{lemma}
\label{lem:VI}
Consider the update rule \eqref{eq:VI_OGD}.
For all $t$, we assume: i)  $F_t$ is $m_t$-strongly-monotone, ii) $F_t$ is $L_t$-Lipschitz, i.e. $\norm{F_t(x)-F_t(\tilde{x})}\le L_t \norm{x-\tilde{x}}$ for all $x,\tilde{x}\in \R^p$, iii) $\norm{v_t}^2 \le \delta_t^2 \norm{F_t(x_t)}^2 + c_t^2$, and  iv) $\|x_{t+1}^*-x_t^*\|\le \sigma_t$. If there exist non-negative scalars $(\rho_t, \lambda_t^{(1)}, \lambda_t^{(2)},\lambda_t^{(3)},\lambda_t^{(4)})$ such that
\begin{align}
\label{eq:LMI_VI}
\bmat{1-\rho_t^2 + \lambda_t^{(4)}L_t^2 -2 \lambda_t^{(1)}  m_t & \lambda_t^{(1)} -\alpha_t & 1  & -\alpha_t \\
\lambda_t^{(1)} -\alpha_t  & \alpha_t^2 -\lambda_t^{(4)} + \delta_t^2 \lambda_t^{(3)} & -\alpha_t & \alpha_t^2\\
1 & -\alpha_t & 1-\lambda_t^{(2) } & -\alpha_t \\
-\alpha_t & \alpha_t^2 & -\alpha_t & \alpha_t^2-\lambda_t^{(3)} }  \preceq 0,
\end{align}
then the iterates of \eqref{eq:VI_OGD} satisfy
$\norm{x_{t+1}-x_{t+1}^*}^2 \le \rho_t^2 \norm{x_t-x_t^*}^2+\lambda_t^{(2)} \sigma_t^2 +\lambda_t^{(3)} c_t^2$.
\end{lemma}
\begin{proof}
Denote $e_t= x_t^* -x_{t+1}^*$. Then \eqref{eq:VI_OGD} can be rewritten as
$x_{t+1}- x_{t+1}^*=x_t-x_t^*-\alpha_t F_t(x_t)+e_t-\alpha_t v_t$, which is a special case of \eqref{eq:gen_algo} with $\xi_t=x_t-x_t^*$, $B_t=\bmat{-\alpha_t & 1 & -\alpha_t}$, and $w_t=\bmat{F_t(x_t)^\tp &  e_t^\tp&  v_t^\tp}^\tp$. Since $F_t$ is assumed to be $m_t$-strongly-monotone, we know $F_t(x_t)^\tp (x_t-x_t^*)\ge m_t\norm{x_t-x_t^*}^2$, 
Denoting $M_t=\bmat{2m_t & -1 \\ -1 & 0}$, we 
obtain a supply rate
 \eqref{eq:gen_ineq} with $\left(X_t^{(1)},\Lambda_t^{(1)} \right) = \left( \bmat{M_t & 0_{2\times 2} \\ 0_{2\times 2} & 0_{2\times 2}} ,0 \right)$.
The assumption $\|x_{t+1}^*-x_t^*\|\le \sigma_t$ leads to the second supply rate condition with $X_t^{(2)}= \diag(0,0,1,0)$ and $\Lambda_t^{(2)}=\sigma_t^2$. The error bound on $v_t$ can be recast as the third supply condition with $X_t^{(3)} = \diag( 0,-\delta_t^2,0,1)$ and $\Lambda_t^{(3)}=c_t^2$. Finally, the assumption that $F_t$ is $L_t$-Lipschitz leads to the fourth supply rate condition with $X_t^{(4)}= \diag(-L_t^2,1,0,0)$ and $\Lambda_t^{(4)}=0$. Then we can directly apply Proposition \ref{prop:main} to obtain the LMI condition~\eqref{eq:LMI_VI}.
\end{proof}
Similarly, we can solve the resultant sequential SDP and obtain the following bound. 
\begin{theorem}\label{th:optimal_param_VI}
Suppose $\alpha_t\le \frac{2(m_t-\delta_t L_t)}{L_t^2 (1-\delta_t^2)}$ and $\delta_t \leq \frac{m_t}{L_t}$ for all $t$. Let $\mathcal{T}_t \subseteq \R_+^5$ be the set of tuples $( \rho_t, \lambda_t^{(1)} , \lambda_t^{(2)}, \lambda_t^{(3)}, \lambda_t^{(4)})$ that satisfy \eqref{eq:LMI_VI}. Suppose the assumptions in Lemma~\ref{lem:VI} hold. Then the bound $\hat{U}_t$ defined in \eqref{Eq:Optimal_U} with $J=4$ is given by
\begin{align}
  \label{Eq:U_VI}
  \hat{U}_{t+1} = 
     \left( \sqrt{(1-2 m_t \alpha_t + \alpha_t^2 L_t^2) \hat{U}_t} +\alpha_t \sqrt{c_t^2+\delta_t^2 L_t^2 \hat{U}_t}+\sigma_t \right)^2.
\end{align}
\end{theorem}
\begin{proof}
The proof still relies on the Schur complement lemma and some algebraic manipulations.
We defer the detailed proof to the appendix.
\end{proof}

From \eqref{Eq:U_VI}, we have 
$\sqrt{{U}_{t+1}}\le (\sqrt{1-2 m_t \alpha_t + \alpha_t^2 L_t^2}+\alpha_t L_t\delta_t)\sqrt{\hat{U}_t}+\alpha_t c_t+\sigma_t$.  The conditions  $\alpha_t\le \frac{2(m_t-\delta_t L_t)}{L_t^2 (1-\delta_t^2)}$ and $\delta_t \leq \frac{m_t}{L_t}$ ensure $\sqrt{1-2 m_t \alpha_t + \alpha_t^2 L_t^2}+\alpha_t L_t\delta_t\le 1$. We can further derive a dynamic regret bound. Such a result will be presented in Section \ref{sec:dynamicregret}.
If $L_t=L$, $m_t=m$, $\delta_t=\delta$, $c_t=c$, $\sigma_t=\sigma$, and $\alpha_t=\alpha$ for all $t$, then we can easily show that $\hat{U}_t$ converges linearly to a small ball as follows:
\begin{align*}
    \hat{U}_t\le (\sqrt{1-2 m \alpha + \alpha^2 L^2}+\alpha L\delta)^t \hat{U}_0+\frac{\alpha c +\sigma}{1-\sqrt{1-2 m \alpha + \alpha^2 L^2}+\alpha L\delta}.
\end{align*}

\begin{remark}\label{remark2}
Our proof for Theorem \ref{th:optimal_param_VI} can be easily modified to show that the recursive bound in \cite[Remark 3]{hu2021analysis} actually gives the exact solution for the sequential SDP in  \cite[Proposition 1]{hu2021analysis}. 
Hence our proof technique bridges the analysis gap in~\cite{hu2021analysis}. More explanations are provided in the appendix.
\end{remark}

\section{Tracking error bounds for online stochastic gradient methods}
\label{sec:Inexact_SGD}
In this section, we consider the case where the inexactness in the online gradient oracle has a stochastic nature.
We will tailor our sequential SDP approach for analyzing the mean-square error of such stochastic online algorithms.  

\subsection{Dissipativity theory for stochastic dynamic systems}
The dissipativity framework presented in Section~\ref{sec:prelim} can be tailored to analyze dynamic systems subject to stochastic noise \cite{hu2018dissipativity,hu2021analysis}.
Consider the dynamic system \eqref{eq:gen_algo} with the input sequence $\{w_t\}$ being a stochastic process.
We can modify Proposition \ref{prop:main} as follows.
\begin{proposition}\label{prop:main_stoc}
Consider the system \eqref{eq:gen_algo} with $\{w_t\}$ being a stochastic process. 
Suppose  $\{\xi_t,w_t\}$ satisfies the following expected supply rate condition for $j=1,\cdots, J$:
\begin{align}\label{eq:exp_supply}
  \mathbb{E} \left( \bmat{ \xi_t \\ w_t }^T (X_t^{(j)}\otimes I_p) \bmat{ \xi_t \\ w_t } \right)\le \Lambda_t^{(j)}.
\end{align}
 If there exist non-negative scalars $\{\lambda_t^{(j)}\}_{j=1}^J$ and $\rho_t$
 such that the LMI condition \eqref{eq:LMI1} holds,
then one must have
$\mathbb{E}\norm{\xi_{t+1}}^2 \le \rho_t^2 \mathbb{E}\norm{\xi_t}^2 + \sum_{j=1}^J \lambda_t^{(j)} \Lambda_t^{(j)}$.
In addition, if \eqref{eq:exp_supply} holds as an equality for $j\in \Omega$ where $\Omega $ is a subset of $\{1,2,\ldots,J\}$, then $\lambda_t^{(j)}$ is allowed to be negative for any $j\in \Omega$ and the LMI condition still holds. 
\end{proposition}
\begin{proof}
From \eqref{eq:LMI1}, we know \eqref{eq:step1} holds almost surely. Then we can take the expectation of both sides of \eqref{eq:step1}, and apply the expected supply rate condition~\eqref{eq:exp_supply} to get the desired conclusion\footnote{Consider the case where \eqref{eq:exp_supply} holds as an equality for $j\in \Omega$. For any $j\in \Omega$, we can allow $\lambda_t^{(j)}$ to be negative, and it is straightforward to verify that the desired conclusion still holds.}.
\end{proof}

Based on the above result, we can still apply the sequential SDP approach in Proposition \ref{Prop:Greedy_app} to obtain mean-square tracking error bounds for online optimization methods subject to stochastic noise. Next, we will present two case studies. 

\subsection{Analysis of stochastic OGD with additive IID noise}
Now we study the performance of the inexact OGD scheme \eqref{eq:inexactOGD} under the alternative assumption that the sequence $\{v_t\}$ is a zero-mean IID process satisfying the mean-square error bound $\mathbb{E}\norm{v_t}^2\le c_t^2$.  
We can apply Proposition~\ref{eq:exp_supply} to obtain the following result.

\begin{lemma}
\label{lem:stoc_OGD} Consider the recursion \eqref{eq:inexactOGD}. For all $t$,
we assume : i)  $f_t \in \mathcal{S}\left(m_t, L_t\right)$, ii) $\|x_{t+1}^*-x_t^*\|\le \sigma_t$, and iii) $\{v_t\}$ is a zero-mean IID process satisfying $\mathbb{E}\norm{v_t}^2\le c_t^2$. If there exist non-negative scalars $(\rho_t, \lambda_t^{(1)}, \lambda_t^{(2)}, \lambda_t^{(3)} )$ and real scalars $(\lambda_t^{(4)},\lambda_t^{(5)}, \lambda_t^{(6)})$ such that the following matrix inequality holds (the definition of $Y_t$ is given by \eqref{Eq:Mu_def})

\begin{align}
\label{eq:LM_SGD}
\bmat{1-\rho^2 & -\alpha_t & 1 & -\alpha_t-\lambda_t^{(4)}\\ -\alpha_t & \alpha_t^2 & -\alpha_t & \alpha_t^2-\lambda_t^{(5)}\\ 1 & -\alpha_t & 1-\lambda_t^{(2)} & -\alpha_t-\lambda_t^{(6)} \\ -\alpha_t-\lambda_t^{(4)} & \alpha_t^2-\lambda_t^{(5)} & -\alpha_t-\lambda_t^{(6)} & \alpha_t^2-\lambda_t^{(3)}}-\lambda_t^{(1)} \bmat{ Y_t& 0_{2\times 2}\\0_{2\times 2} & 0_{2\times 2}}\preceq 0,
\end{align}
then we must have
$\mathbb{E}\norm{x_{t+1}-x_{t+1}^*}^2 \le \rho_t^2 \mathbb{E}\norm{x_t-x_t^*}^2+\lambda_t^{(2)} \sigma_t^2 +\lambda_t^{(3)} c_t^2$.
\end{lemma}
\begin{proof}
Set $\xi_t=x_t-x_t^*$, $w_t=\bmat{ \nabla f_t(x_t)^\tp &  e_t^\tp &  v_t^\tp}^\tp$, and  $B_t=\bmat{-\alpha_t & 1 & -\alpha_t}$. 
The assumptions (i) and (ii) allow us to specify the first two expected supply rate conditions with 
 $(X_t^{(1)},\Lambda_t^{(1)})=(\diag(Y_t,0,0),0)$ and $(X_t^{(2)},\Lambda_t^{(2)})=(\diag(0,0,1,0),\sigma_t^2)$. The mean-square error bound $\mathbb{E} \|v_t\|^2\le c_t^2$ can be rewritten as another expected supply rate condition with $(X_t^{(3)},\Lambda_t^{(3)})=(\diag(0,0,0,1),c_t^2)$. Based on the fact that $\{v_t\}$ is a zero-mean IID process, we have $\mathbb{E} \left( v_t^\tp (x_t-x_t^*) \right)=0$, $ \mathbb{E} \left( v_t^\tp \nabla f_t(x_t)\right) =0$, and  $\mathbb{E} \left( v_t^\tp e_t\right)=0$. These three conditions lead to the expected supply rate conditions in the equality form and can be combined with the first three supply rate conditions to derive the desired LMI condition.
  Notice that $(\lambda_t^{(4)},\lambda_t^{(5)},\lambda_t^{(6)})$ are allowed to be negative since the associated supply rate conditions are in the equality form.
\end{proof}

Next, we can apply the sequential SDP approach to obtain the following recursive mean-square tracking error bound\footnote{In this case, we will have $\mathbb{E}\norm{x_t-x_t^*}^2\le \hat{U}_t$, where $\hat{U}_t$ is solved from the sequential SDP problem in Proposition \ref{Prop:Greedy_app}.}. 

\begin{theorem}\label{th:stoc_OGD_param}
Suppose $\alpha_t\le \frac{2}{L_t}$. Let $\mathcal{T}_t \subseteq \R_+^4\times \R^3$ denote the set of tuples $\left( \rho_t, \lambda_t^{(1)} , \hdots , \lambda_t^{(6)} \right) $ that satisfy the LMI condition \eqref{eq:LM_SGD} .  If the assumptions in \eqref{lem:stoc_OGD} hold, then the bound $\hat{U}_t$ defined in \eqref{Eq:Optimal_U} with $J=6$  is exactly given by
\begin{align}
  \label{Eq:U_SOGD_abs_full}
  \hat{U}_{t+1} = 
     \left( \mu_t \sqrt{\hat{U}_t}+\sigma_t \right)^2 + \alpha_t^2 c_t^2. 
\end{align}
\end{theorem}
\begin{proof}
Based on Proposition \ref{Prop:Greedy_app}, we know that $\hat{U}_{t+1}$ can be calculated as
\begin{equation}\label{eq:SGD_Opt_stg1a}
\hat{U}_{t+1}= \minimize_{(\rho_t,\lambda_t^{(1)},\hdots,\lambda_t^{(6)})\in\mathcal{T}_t}
 \rho_t^2 \hat{U}_t +\lambda_t^{(2)} \sigma_t^2 +\lambda_t^{(3)} c_t^2.
\end{equation}
For any $(\rho_t,\lambda_t^{(1)},\cdots,\lambda_t^{(6)})\in\mathcal{T}_t$, the $3\times 3$ upper-left block of the left side of \eqref{eq:LM_SGD} must be negative semidefinite, and hence we know
\eqref{eq:LMI_OGD} holds. Let $\hat{\mathcal{T}}_t\subset \R_+^4\times \R^3$ denote the set of tuples $\left( \rho_t, \lambda_t^{(1)} , \cdots , \lambda_t^{(6)} \right) $ that satisfy the LMI condition \eqref{eq:LMI_OGD}. Then we know $\mathcal{T}_t\subset \hat{\mathcal{T}}_t$. Noticing $\lambda_t^{(3)}\ge \alpha_t^2$ for any point in $\mathcal{T}_t$, we can make the following~claim:
 \begin{align*}
\hat{U}_{t+1} &\ge \minimize_{(\rho_t, \lambda_t^{(1)}, \cdots, \lambda_t^{(6)}) \in \mathcal{T}_t} \rho_t^2 \hat{U}_t+\lambda_t^{(2)} \sigma_t^2 +\alpha_t^2 c_t^2\\ &\ge \minimize_{(\rho_t, \lambda_t^{(1)}, \cdots, \lambda_t^{(6)}) \in \hat{\mathcal{T}}_t}  \rho_t^2 \hat{U}_t+\lambda_t^{(2)} \sigma_t^2 +\alpha_t^2 c_t^2\\
&= \minimize_{(\rho_t, \lambda_t^{(1)},\lambda_t^{(2)}) \in \hat{\mathcal{V}}_t}  \rho_t^2 \hat{U}_t+\lambda_t^{(2)} \sigma_t^2 +\alpha_t^2 c_t^2\\
&= \left( \mu_t \sqrt{\hat{U}_t}+\sigma_t \right)^2 + \alpha_t^2 c_t^2,
 \end{align*}
where $\hat{\mathcal{V}}_t\subset \R_+^3$ denotes the set of tuples $\left( \rho_t, \lambda_t^{(1)}, \lambda_t^{(2)} \right) $ that satisfy the LMI condition~\eqref{eq:LMI_OGD}, and the last step can be easily verified using the argument in Section \ref{sec:Inexact_OGD_Abs}. It is also straightforward to verify that 
$
\rho_t^2 =\mu_t^2 (1+\nu_t), \,\,
\lambda_t^{(1)}=\frac{\mu_t \alpha_t(1+\nu_t)}{L_t-m_t}, \,\,
\lambda_t^{(2)}= \alpha_t^2, \,\,
\lambda_t^{(3)}= 1+\frac{1}{\nu_t}, \,\,
\lambda_t^{(4)}= -\alpha_t, \,\,
\lambda_t^{(5)}=  \alpha_t^2, \,\,
\lambda_t^{(6)}=  -\alpha_t
$  is a feasible point in $\mathcal{T}_t$.
For the above feasible point, we have $\rho_t^2 \hat{U}_t+\lambda_t^{(2)} \sigma_t^2 +\alpha_t^2 c_t^2=\left( \mu_t \sqrt{\hat{U}_t}+\sigma_t \right)^2 + \alpha_t^2 c_t^2$. Therefore, we also have $\hat{U}_{t+1}\le \left( \mu_t \sqrt{\hat{U}_t}+\sigma_t \right)^2 + \alpha_t^2 c_t^2$. Combining the upper and lower bounds for $\hat{U}_{t+1}$, we reach the desired conclusion.
\end{proof}
\begin{remark}\label{rem:4}
We can see that the bound in \eqref{Eq:U_SOGD_abs_full} is always smaller than or equal to the bound in \eqref{Eq:U_inexactOGD_abs_full}. Obviously, the stochastic assumption on $\{v_t\}$ allows us to derive a refined bound. 
If we have $L_t=L$, $m_t=m$, $\alpha_t=\alpha$, $\sigma_t=\sigma$ and $c_t=c$ for all $t$, then \eqref{Eq:U_SOGD_abs_full} becomes $\hat{U}_{t+1}= ( \mu \sqrt{\hat{U}_t}+\sigma )^2 + \alpha^2 c^2 $. Then it seems that both \eqref{Eq:U_SOGD_abs_full} and \eqref{Eq:U_inexactOGD_abs_full} lead to the bound
 $\sqrt{\hat{U}_t} \leq \mu^t \sqrt{\hat{U}_0} + \frac{\sigma + \alpha c}{1-\mu} $, which states that $\sqrt{\hat{U}_t}$ converges linearly to an asymptotic value $\frac{\alpha c}{1-\mu} $ at the rate $\mu$.
However, we can use \eqref{Eq:U_SOGD_abs_full} to derive a different bound with an improved asymptotic value. Noticing $0\le \mu \le 1$, the fixed point of \eqref{Eq:U_SOGD_abs_full} is unique and satisfies $U^*=(\mu\sqrt{U^*}+\sigma)^2+\alpha^2 c^2$. We have $\sqrt{U^*}= \frac{ \mu \sigma +\sqrt{\sigma^2 +\alpha^2 c^2 (1- \mu^2 )}}{1-\mu^2}$.
Notice that the right side of \eqref{Eq:U_SOGD_abs_full} is concave in $\hat{U}_t$, and hence we have
$\hat{U}_{t+1}\le U^*+ \left(\mu +\frac{\sigma}{\sqrt{U^*}}\right)(\hat{U}_t-U^*) \leq \left( \mu +\frac{\sigma }{\sqrt{U^*}}\right)^{t+1}\hat{U}_0 +U^*$.  
Since we know $(\mu+\sigma/\sqrt{U^*})^2=1-\alpha^2 c^2 /U^*\le 1$, we have obtained an alternative bound which gives the true value $U^*$ at the price of yielding a slower rate. 
\end{remark}

\subsection{Analysis of stochastic OGD in the finite-sum setting}
For many machine learning tasks, the cost function has a finite-sum structure. For instance, the total cost in supervised learning can be typically decomposed as a sum of the loss on different data points in the training set. If the data points used for training are time-varying, then we have
$f_t=\frac{1}{n}\sum_{i=1}^n f_t^{(i)}$, where $n$ is the size of the training set. Instead of utilizing the full gradient information, one can randomly sample one data point from the training set and use the following stochastic online gradient scheme:
\begin{align}\label{eq:SOGD}
x_{t+1}=x_t-\alpha \nabla f_t^{(i_t)}(x_t),
\end{align}  
where $i_t\in \{1,2,\cdots,n\}$ is an index sampled in an IID manner. In such a setting, we can still apply Proposition \ref{prop:main_stoc} to obtain the following LMI condition.

\begin{lemma}
\label{lem:SGD_sum}
Consider the recursion \eqref{eq:SOGD}.
For all $t$, we assume: i)  $f_t$ is $m_t$-strongly convex, and ii) $\|x_{t+1}^*-x_t^*\|\le \sigma_t$. Denote $G_t=\frac{1}{n}\sum_{i=1}^n \norm{\nabla f_t^{(i)}(x_t^*)}^2$.  Based on the assumption on $f_t^{(i)}$, we define $\tilde{X}_t\in \R^{2\times 2}$ as follows.
\begin{align}
  \label{Eq:X_def}
  \tilde{X}_t :=  \left\{
    \begin{array}{ll}
      \bmat{2L_t^2 & 0\\ 0 & -1} & \mbox{if } f_t^{(i)} \mbox{\,is\,\,\,} L_t\mbox{-smooth}\\
      \bmat{0 & L_t \\ L_t & -1} & \mbox{if } f_t^{(i)} \mbox{\,is\,\,\,} L_t \mbox{-smooth and convex} \\
     \bmat{-2L_t m_t & L_t+m_t\\ L_t+m_t & -1} &  \mbox{if } f_t^{(i)} \mbox{\,is\,\,\,} L_t\in\mathcal{F}(m_t,L_t)
    \end{array}
  \right..
\end{align}
If there exists non-negative scalars $(\rho_t, \lambda_t^{(1)}, \lambda_t^{(2)},\lambda_t^{(3)})$ such that
\begin{align}
\label{Eq:LMI_SGD_fin_sum}
\bmat{1-\rho_t^2 & -\alpha_t & 1 \\ -\alpha_t & \alpha_t^2 & -\alpha_t \\ 1 & -\alpha_t & 1-\lambda_t^{(2)}}+\lambda_t^{(1)} \bmat{-2m_t & 1 & 0  \\ 1 & 0 & 0  \\ 0 & 0 & 0 }+\lambda_t^{(3)}\bmat{\tilde{X}_t & 0_{2\times 1} \\ 0_{1\times 2} & 0}\preceq 0,
\end{align}
then the iterates of \eqref{eq:SOGD} satisfy 
$\mathbb{E}\norm{x_{t+1}-x_{t+1}^*}^2 \le \rho_t^2 \mathbb{E}\norm{x_t-x_t^*}^2+\lambda_t^{(2)} \sigma_t^2+2\lambda_t^{(3)} G_t^2$.
\end{lemma}
\begin{proof}
We can set $e_t=x_t^*-x_{t+1}^*$ and rewrite \eqref{eq:SOGD} as
$x_{t+1}-x_{t+1}^*=x_t-x_t^*-\alpha \nabla f_t^{(i_t)}(x_t)+e_t$,
which becomes a special case of \eqref{eq:gen_algo} if we set $\xi_t= x_t-x_t^*$, $B_t=\bmat{-\alpha_t & 1 }$, and $w_t=\bmat{\nabla f_t^{(i_t)}(x_t)^\tp  & e_t^\tp}^\tp$.  The $m_t$-strong convexity of $f_t$ leads to a supply rate condition with $(X_t^{(1)},\Lambda_t^{(1)})=(\diag(M_t,0),0)$, where $M_t$ is given by \eqref{Eq:Mu_def}. The assumption $\norm{e_t} \le \sigma_t$ leads to the second supply rate condition with $X_t^{(2)}= \diag(0,0,1)$ and $\Lambda_t^{(2)}=\sigma_t^2$. The assumption on $f_t^{(i)}$ can be translated to the third supply rate condition with $X_t^{(3)}=- \diag(\tilde{X}_t,0)$ and $\Lambda_t^{(3)}=2G_t^2$. Combining these supply rate conditions with Proposition \ref{prop:main_stoc} immediately leads to the desired~LMI.
\end{proof}
Incorporating \eqref{Eq:LMI_SGD_fin_sum} into our sequential SDP framework, we obtain the following result. 
\begin{theorem}\label{th:SGD_fin_sum_param}
Suppose $\alpha_t\in [0, \bar{\alpha}_t]$, where $\bar{\alpha}_t$ is defined as
\begin{align*}
  \bar{\alpha}_t :=  \left\{
    \begin{array}{ll}
      \frac{m_t}{L_t^2} & \mbox{if } f_t^{(i)} \mbox{\,is\,\,\,} L_t\mbox{-smooth}\\
      \frac{1}{L_t} & \mbox{if } f_t^{(i)} \mbox{\,is\,\,\,} L_t \mbox{-smooth and convex} \\
     \frac{1}{L_t+m_t} &  \mbox{if } f_t^{(i)} \mbox{\,is\,\,\,} L_t\mbox{-smooth and $m_t$-strongly convex}
    \end{array}
  \right..
\end{align*}
Let $\mathcal{T}_t \subseteq \R_+^4$ denote the set of tuples $\left( \rho_t, \lambda_t^{(1)} , \lambda_t^{(2)}, \lambda_t^{(3)}\right) $ that satisfy \eqref{Eq:LMI_SGD_fin_sum} . Suppose the assumptions in Lemma~\ref{lem:SGD_sum} hold. Then  $\hat{U}_t$ from \eqref{Eq:Optimal_U} with $J=3$ is given by
\begin{align}
  \label{Eq:U_SOGD_sum_full}
  \hat{U}_{t+1} = 
     \left( \sqrt{\hat{\rho}_t \hat{U}_t +2 \alpha_t^2 G_t^2} +\sigma_t \right)^2,
\end{align}
where $\hat{\rho}_t=1-2 \alpha_t m_t +\tilde{m} \alpha_t^2$ and $\tilde{m}_t:=\tilde{X}_t(1,1)+2 m_t \tilde{X}_t(2,1)$.
\end{theorem}
\begin{proof}
Based on Proposition \ref{Prop:Greedy_app}, we need to solve the following problem:
\begin{equation}\label{Eq:SGD_Fsum_Opt_stg1}
\hat{U}_{t+1}= \minimize_{(\rho_t,\lambda_t^{(1)},\lambda_t^{(2)},\lambda_t^{(3)}  )\in\mathcal{T}_t} \rho_t^2 \hat{U}_t + \lambda_t^{(2)} \sigma_t^2 +2\lambda_t^{(3)} G_t^2.
\end{equation}
Set $\lambda_t^{(2)}=1+\frac{1}{\nu_t}$. We can apply the Schur complement to convert \eqref{Eq:LMI_SGD_fin_sum} to the equivalent condition, i.e. $\nu_t >0, \,\, \lambda_t^{(1)}\ge 0,  \,\, \lambda_t^{(3)}\ge 0$ and  :
\[
\bmat{\left(1+\nu_t \right) -\rho_t^2-2 m_t \lambda_t^{(1)} + \lambda_t^{(3)} \tilde{X}_{t}(1,1) & \lambda_t^{(1)} +\lambda_t^{(3)}\tilde{X}_t(1,2)-\alpha_t \left( 1 + \nu_t \right) \\
\lambda_t^{(1)} +\lambda_t^{(3)}\tilde{X}(1,2)-\alpha_t \left( 1 + \nu_t \right)  & \alpha_t^2 \left(1+\nu_t \right)+\lambda_t^{(3)}\tilde{X}(2,2)   } \preceq 0.
\]
Noticing $\tilde{X}_t(2,2)=-1$ from (\ref{Eq:X_def}), we set $\tau_t=\lambda_t^{(3)} -\alpha_t^2 \left(1+\nu_t\right)$, and this leads to $\lambda_t^{(3)}=\tau_t +\alpha_t^2 \left(1+\nu_t\right)$. Based on the Schur complement lemma, we can set $\phi_t:=\left(1+\nu_t\right) \left(1 -\alpha_t \tilde{X}_t(2,1)\right)$ and show that the above condition is equivalent to $\nu_t >0$, $\tau_t \geq 0$, $\lambda_t^{(1)}\ge 0$, and $h(\nu_t,\tau_t,\lambda_t^{(1)}) \leq \rho_t^2$, where $h(\nu_t,\tau_t,\lambda_t^{(1)})$ is defined as
\begin{align*}
h(\nu_t,\tau_t,\lambda_t^{(1)})= \left(1+\nu_t \right) & \left(1+\alpha_t^2 \tilde{X}_t(1,1)\right) +2\lambda_t^{(1)}\left( \tilde{X}_t(2,1)-m_t\right) \\+ & \tau_t\left(\tilde{X}_t(1,1)+(\tilde{X}_t(2,1))^2 \right)+ 
 \frac{ (  \lambda_t^{(1)} -\alpha_t \phi_t )^2}{\tau_t} -2\alpha_t \phi_t \tilde{X}_t(2,1).
\end{align*}
Next, we want to show $h(\nu_t,\tau_t,\lambda_t^{(1)}) \ge 0$. We need to consider three different cases corresponding to various assumptions on $f_t^{(i)}$ used in \eqref{Eq:X_def}. 

In the first case, we assume $f_t^{(i)}$ is only $L_t$-smooth. We have $\tilde{X}_t(1,1)=2L_t^2$ and $\tilde{X}_t(2,1)=0$
For any fixed $\nu_t>0$ and $\tau_t \ge 0$, we have
\[
h(\nu_t,\tau_t,\lambda_t^{(1)})= \left(1+\nu_t \right)  \left(1+\alpha_t^2 \tilde{X}_t(1,1)\right) -2\lambda_t^{(1)}m_t +2L_t^2 \tau_t+\frac{ (  \lambda_t^{(1)} -\alpha_t(1+\nu_t))^2}{\tau_t}, 
\]
which is non-negative based on the following fact:
\begin{align*}
h(\nu_t, \tau_t, \lambda_t^{(1)}) \ge h(\nu_t, \tau_t, \alpha_t \phi_t +\tau_t m_t) 
= \left(1+\nu_t \right)  \left(1-2m_t\alpha_t+2L_t^2 \alpha_t^2\right) + \tau_t(2L_t^2 -m_t^2).
\end{align*}
 In the second case, we assume $f_t^{(i)}$ is $L_t$-smooth and convex. We have $\tilde{X}_t(1,1)=0$ and $\tilde{X}_t(2,1)=L_t$.
For any fixed $\nu_t>0$ and $\tau_t \ge 0$, we have:
 \[
 h(\nu_t,\tau_t,\lambda_t^{(1)})= 1+\nu_t    +2\lambda_t^{(1)}\left( L_t-m_t\right) +\tau_t L_t^2+\frac{ (  \lambda_t^{(1)} -\alpha_t \phi_t )^2}{\tau_t} -2\alpha_t \phi_t L_t.
   \]
We can prove that $h(\nu_t,\tau_t,\lambda_t^{(1)})$ has a non-negative lower bound given below by $\hat{h}(\nu_t, \tau_t)$ for all $\lambda_t^{(1)}$:
\begin{align*}
    \hat{h}(\nu_t, \tau_t):= \begin{cases}
\left(1+\nu_t \right)\left( 1+2\alpha_t m _t \left( \alpha_t L_t - 1 \right) \right) + \tau_t m_t(2 L_t - m_t)  
 & \text{if} \,\ 0 \le \tau_t \le \frac{\alpha_t\phi_t}{L_t-m_t} \\  
1+\nu_t +\tau_t L_t^2+\frac{ \alpha_t^2 \phi_t^2}{\tau_t} -2\alpha_t \phi_t L_t  & \text{if} \,\  \tau_t > \frac{\alpha_t\phi_t}{L_t-m_t} 
\end{cases}
\end{align*}
Since we know $\frac{\alpha_t \phi_t}{L_t} \le \frac{\alpha_t \phi_t}{L_t-m_t}$ for any fixed $\nu_t$, we can easily conclude  that $\hat{h}(\nu_t,\tau_t)$ linearly increases for $0 \le \tau_t \le \frac{\alpha_t\phi_t}{L_t-m_t}$ and then again monotonically increases for $\tau_t > \frac{\alpha_t\phi_t}{L_t-m_t} $. Therefore $\hat{h}$ achieves its minimum at $\tau_t=0$. Therefore, we must have $h(\nu_t, \tau_t, \lambda_t^{(1)}) \ge h(\nu_t, \tau_t, \alpha_t(1+\nu_t)(1-\alpha_t L_t)-\tau_t(L_t-m_t))= \hat{h}(\nu_t, \tau_t)\ge 0$.

 In the third case, $f_t^{(i)}$ is assumed to be $L_t$-smooth and $m_t$-strongly convex. We have $\tilde{X}_t(1,1)=-2 L_t m_t$ and $\tilde{X}_t(2,1)=L_t+m_t$.
For any $\nu_t>0$ and $\tau_t \ge 0$, we~have
 \begin{align*}
       h(\nu_t,\tau_t,\lambda_t^{(1)})= \left(1+\nu_t \right)  \left(1-2\alpha_t^2 L_t m_t\right) +  & 2\lambda_t^{(1)}L_t +\tau_t\left(L_t^2+m_t^2\right) \\
       & +\frac{(\lambda_t^{(1)} -\alpha_t \phi_t )^2}{\tau_t} -2\alpha_t \phi_t (L_t+m_t).
 \end{align*}
Again, we can prove that $h(\nu_t,\tau_t,\lambda_t^{(1)})$ has a non-negative lower bound $\hat{h}(\nu_t, \tau_t)$ for all $\lambda_t^{(1)}$. Specifically, we set $\hat{h}(\nu_t,\tau_t)=\left(1+\nu_t \right)\left( 1-2\alpha_t m_t+2\alpha_t^2 m_t^2 \right) +\tau_t m_t^2 $
for $ 0 \le \tau_t \le \frac{\alpha_t\phi_t}{L_t}$, and set $\hat{h}(\nu_t,\tau_t)=\left(1+\nu_t \right) \left(1- \right. \left. 2\alpha_t^2 L_t m_t\right)+\tau_t\left(L_t^2+m_t^2\right)  +\frac{ \alpha_t^2 \phi_t^2}{\tau_t} -2\alpha_t \phi_t (L_t+m_t)$
for $\tau_t>\frac{\alpha_t \phi_t}{L_t}$. Since we have $\frac{\alpha_t \phi_t}{L_t^2+m_t^2} \le \frac{\alpha_t \phi_t}{L_t}$ for any fixed $\nu_t$, we can conclude that $\hat{h}(\nu_t,\tau_t)$ linearly increases for $0 \le \tau_t \le \frac{\alpha_t\phi_t}{L_t}$ and then again monotonically increases for $\tau_t > \frac{\alpha_t\phi_t}{L_t} $. Therefore $\hat{h}$ achieves its minimum at $\tau_t=0$. This lead us to the fact that $h(\nu_t, \tau_t, \lambda_t^{(1)}) \ge h(\nu_t, \tau_t, \alpha_t(1+\nu_t)(1-\alpha_t (L_t+m_t)-\tau_t L_t)= \hat{h}(\nu_t, \tau_t)\ge 0$.

Now we know $h(\nu_t,\tau_t,\lambda_t^{(1)})\ge 0$. Denote $F(\nu_t,\tau_t):= \left( 1+\frac{1}{\nu_t}\right) \sigma_t^2  +2\left(\tau_t +\alpha_t^2\left( 1+\nu_t \right) \right) G_t^2 $, and the optimization problem \eqref{Eq:SGD_Fsum_Opt_stg1} can be solved as follows
\begin{align*}
    \hat{U}_{t+1}&=\min_{\nu_t>0}\,\,\left( \min_{\tau_t \ge 0} \,\,\left( \min_{\lambda_t^{(1)}\ge 0} \,\,\left( \min_{\rho^2_t\ge h\left(\nu_t,\tau_t,\lambda_t^{(1)} \right) }  \left(\rho_t^2 \hat{U}_t + F(\nu_t,\tau_t)\right) \right)\right) \right)\\
    &=\min_{\nu_t>0}\,\,\left( \min_{\tau_t \ge 0} \,\,\left( \min_{\lambda_t^{(1)}\ge 0} \, \left( h(\nu_t,\tau_t,\lambda_t^{(1)})\hat{U}_t + F(\nu_t,\tau_t) \right) \right)\right)\\
     &=\left( \sqrt{\hat{\rho}_t  \hat{U}_t +2 \alpha_t^2 G_t^2} +\sigma_t \right)^2.
\end{align*}
 Notice the optimal $\lambda_t^{(1)}$ is given by $\alpha_t (1+\nu_t)(1-\alpha_t \tilde{X}_t(2,1))$. Since $\lambda_t^{(1)} \ge 0$ and $\nu_t >0$, therefore the derived solution of the above optimization holds if $0 \le \alpha_t \tilde{X}_t(2,1) \le 1$ which is consistent with the existing result in \cite{hu2021analysis}. Additionally, notice that $\rho_t \le 1$ and $0 \le \alpha_t \tilde{X}_t(2,1) \le 1$ lead to the desired range of $\alpha_t$ for which this result holds. 
\end{proof}

If we assume $L_t=L$, $m_t=m$, $G_t=G$, $\sigma_t=\sigma$, and $\alpha_t=\alpha$ for all $t$, we can simplify \eqref{Eq:U_SOGD_sum_full} as $\sqrt{\hat{U}_{t+1}}= \sqrt{ \hat{\rho}\hat{U}_t +2 \alpha^2 G^2} +\sigma $. This leads to the bound $\sqrt{\hat{U}_t} \leq \hat{\rho}^\frac{t}{2} \sqrt{\hat{U}_0} +\frac{\sqrt{2} \alpha G +\sigma}{1-\sqrt{\hat{\rho}}}$, which shows that $\sqrt{\hat{U}_t}$ converges below a steady state value $\frac{\sqrt{2} \alpha G +\sigma}{1-\sqrt{\hat{\rho}}}$ at a linear rate $\hat{\rho}^\frac{1}{2}$.
We can also use the argument in Remark \ref{rem:4} to get an alternative bound incoporated with the true steady-state value limit $U^*$ satisfying $\sqrt{U^*}=\sqrt{\hat{\rho} U^*+2\alpha^2 G^2}+\sigma$. 
The details are omitted.

\section{New Bounds for Inexact Proximal Online Gradient Descent}
\label{sec:Inexact_IPOGD}
The inexact proximal online gradient descent (IP-OGD) method is proposed in \cite{dixit2019online} and has shown great promise in addressing nonsmooth time-varying optimization problems such as robust subspace tracking. In this setting, the cost at time $t$ is given by a sum $f_t+g_t$, where $f_t$ is a smooth convex function and $g_t$ is non-smooth and convex. 
For each $t$, the IP-OGD method (or equivalently Algorithm 1 in \cite{dixit2019online}) iterates as follows
\begin{align}\label{Eq:InExact_Proximal_Grad_main}
x_{t+1}=\argmin_{x\in \R^p} \left\{\frac{1}{2\alpha}\norm{x-x_t+\alpha \left(\nabla f_t(x_t) +v_t\right)}^2+g_t(x)\right\},
\end{align}
where the term $v_t$ captures the inexactness in the gradient computation and is assumed to satisfy some prescribed norm bound. 
The assumption adopted in \cite{dixit2019online} is $\mathbb{E}\norm{v_t}\le c_t$. 
Consequently, the tracking error bounds derived in \cite{dixit2019online} hold in $L_1$ sense, i.e. upper bounds are obtained for $\mathbb{E}\norm{x_t-x_t^*}$. 
In this paper, we will adopt a more standard mean-squared error bound for $v_t$, i.e. we assume $\mathbb{E}\norm{v_t}^2\le c_t^2$. We are interested in deriving upper bounds for the mean-square tracking error $\mathbb{E}\norm{x_t-x_t^*}^2$.
Our tracking error bounds hold in the $L_2$ sense, and hence naturally complements the $L_1$ error bounds in \cite{dixit2019online}.
Other than this one key difference in the assumption for $v_t$, all other assumptions in our analysis will be the same as the ones used in \cite{dixit2019online}.
We can obtain the following LMI condition.

 \begin{lemma}
\label{lem:inexact_prox_grad}
For all $t$, we assume: i)  $f_t\in \mathcal{F}(m_t,L_t)$, ii) $g_t$ is convex and $L_g$-Lipschitz, iii)
$\|x_{t+1}^*-x_t^*\|\le \sigma_t$, and iv) $\mathbb{E}\norm{v_t}^2\le c_t^2$. If there exist non-negative scalars $\rho_t$ and $\{\lambda_t^{(j)}\}_{j=1}^5$  such that the following matrix inequality holds
\begin{align*}
M(Q_t,S_t,R_t)=\bmat{Q_t & S_t \\ S_t^\tp & R_t}\preceq 0,
\end{align*}
where $(Q_t, S_t, R_t)$ are defined as
\begin{align*}
    Q_t&=\bmat{1-\rho_t^2-2\lambda_t^{(1)}L_t m_t & \lambda_t^{(1)}(L_t + m_t)-\alpha_t \\ \lambda_t^{(1)}(L_t + m_t)-\alpha_t & \alpha_t^2 -2 \lambda_t^{(1)}},\\
    S_t&=\bmat{\lambda_t^{(5)}-\alpha_t & -\alpha_t & -\alpha_t & 1\\ \alpha_t^2-\alpha_t\lambda_t^{(5)}& \alpha_t^2 & \alpha_t^2 & -\alpha_t},\\
    R_t&=\bmat{ \alpha_t^2 -2\alpha_t \lambda_t^{(5)} & \alpha_t^2 -\alpha_t \lambda_t^{(5)} & \alpha_t^2 -\alpha_t \lambda_t^{(5)} & \lambda_t^{(5)} -\alpha_t \\ 
  \alpha_t^2 -\alpha_t \lambda_t^{(5)} & \alpha_t^2-\lambda_t^{(4)} & \alpha_t^2 & -\alpha_t \\ 
  \alpha_t^2-\alpha_t \lambda_t^{(5)}  & \alpha_t^2 & \alpha_t^2-\lambda_t^{(3)} & -\alpha_t \\
  \lambda_t^{(5)}-\alpha_t & -\alpha_t & -\alpha_t &  1-\lambda_t^{(2)}},
\end{align*}
then the IP-OGD method satisfies the following mean-squared tracking error bound:
\[
\mathbb{E}\norm{x_{t+1}-x_{t+1}^*}^2 \le \rho_t^{2} \mathbb{E}\norm{x_t-x_t^*}^2+\lambda_t^{(2)} \sigma_t^2+\lambda_t^{(3)} c_t^2 +4 \lambda_t^{(4)} L_g^2.
\]
\end{lemma}
\begin{proof}
The IP-OGD recursion
can be rewritten as $x_{t+1}=x_t-\alpha \left(\nabla f_t(x_t) +v_t\right)-\alpha r_t$, where 
$r_t\in \partial g_t(x_{t+1})$. Denote $e_t=x_t^*-x_{t+1}^*$, and \eqref{Eq:InExact_Proximal_Grad_main} can be rewritten as
$x_{t+1}-x_{t+1}^*=x_t-x_t^*-\alpha (\nabla f_t(x_t)-\nabla f_t(x_t^*))-\alpha (r_t-r_{t+1}^*)-\alpha (r_{t+1}^*-r_t^*)-\alpha v_t+e_t$,
where $r_t^*$ is a subgradient of $g_t$ evaluated at $x_t^*$ and satisfying $r_t^*=-\nabla f_t(x_t^*)$, and $r_{t+1}^*$ is any subgradient of $g_t$ evaluated at $x_{t+1}^*$. 
Set $\xi_t:=x_t-x_t^*$. Then the IP-OGD method is a special case of \eqref{eq:gen_algo} with $B_t$ and $w_t$ given as 
$B_t=\bmat{-\alpha_t & -\alpha_t & -\alpha_t & -\alpha_t   &  1}$,
and $w_t =\bmat{\nabla f_t(x_t)^\tp-\nabla f_t(x_t^*)^\tp & r_t^\tp-r_{t+1}^{*\tp} & r_{t+1}^{*\tp}-r_t^{*\tp} & v_t^\tp  & e_t^\tp   }^\tp$.
Next, we provide five different expected supply rate conditions to capture the assumptions in the above lemma statement. 
We need to specify
$(X_t^{(j)}, \Lambda_t^{(j)})$ for $j=1,2,\cdots,5$. The assumption $f_t \in \mathcal{F}(m_t,L_t)$ leads to an expected supply rate condition in the form of ~\eqref{eq:exp_supply} with $(X_t^{(1)}, \Lambda_t^{(1)})=\left( \diag(Y_t ,0,0,0,0) ,0 \right)$, where $Y_t$ is given by \eqref{Eq:Mu_def}.
The condition 
$\|e_t\|\le \sigma_t$ can be rewritten in the quadratic form \eqref{eq:exp_supply} with $ (X_t^{(2)}, \Lambda_t^{(2)})=\left( \diag(0,0,0,0,0,1), \sigma_t^2 \right)$.
It is also straightforward to rewrite $\mathbb{E}\norm{v_t}^2\le c_t^2$ in the form of \eqref{eq:exp_supply} with $ (X_t^{(3)}, \Lambda_t^{(3)})=\left( \diag(0,0,0,0,1,0),c_t^2\right)$.
Since $g_t$ is $L_g$-Lipschitz for all $t$, we have
 $\norm{r_t^*-r_{t+1}^*}\le 2L_g$ which can be recast in the form of \eqref{eq:exp_supply} with $(X_t^{(4)}, \Lambda_t^{(4)})=\left( \diag(0,0,0,1,0,0),  4L_g^2 \right)$. The last supply rate condition is a direct consequence of the convexity of $g_t$, and the condition $(r_t-r_{t+1}^*)^\tp (x_{t+1}-x_{t+1}^*)\ge 0$ can be used to derive a supply rate condition with the following parameters: 
\[
X_t^{(5)}= \bmat{0 & 0 & -1 & 0 & 0 & 0 \\ 0 & 0 & \alpha_t & 0 & 0 & 0 \\ -1 & \alpha_t & 2\alpha_t & \alpha_t & \alpha_t & -1 \\ 0 & 0 &\alpha_t & 0 & 0 & 0\\ 0 & 0 & \alpha_t & 0 & 0 & 0\\ 0 & 0 & -1 & 0 & 0 & 0 } ,  \,\,\mbox{and}\,\,\Lambda_t^{(5)}=0
\]
Then we can apply Proposition \ref{prop:main_stoc} to formulate the desired LMI.
\end{proof}
 Applying our sequential SDP approach, we can obtain the following mean-square tracking error bound for IP-OGD. 

\begin{theorem}\label{th:inexact_prox_grad_param}
Suppose $\alpha_t\le \frac{2}{L_t}$ for all $t$. Let $\mathcal{T}_t \subseteq \R_+^6$ denote the set of tuples $\left( \rho_t, \lambda_t^{(1)},  \hdots , \lambda_t^{(5)} \right) $ that satisfy $M(Q_t,S_t,R_t) \preceq 0$. Suppose the assumptions in Lemma~\ref{lem:inexact_prox_grad} hold, then the bound $\hat{U}_t$ defined in \eqref{Eq:Optimal_U} can be calculated as
\begin{align}
  \label{Eq:U_inexact_proxGrad}
  \hat{U}_{t+1} = 
    \left( \mu_t \sqrt{\hat{U}_t}+ \alpha_t \sigma_t + c_t + 2 \alpha_t L_g\right)^2.  
\end{align}
\end{theorem}
\begin{proof}
We need to solve the following SDP:
\begin{equation}\label{eq:InEx_ProxGrad_stg1}
\hat{U}_{t+1}= \minimize_{(\rho_t, \lambda_t^{(1)},  \hdots , \lambda_t^{(5)}) \in \mathcal{T}_t } 
\rho_t^{2} \hat{U}_t +\lambda_t^{(2)} \sigma_t^2+\lambda_t^{(3)} c_t^2 +4 \lambda_t^{(4)} L_g^2.
\end{equation}
We can modify the previous arguments and apply the Schur complement lemma to convert the above optimization problem into the following equivalent form:
\begin{align}\label{eq:IPOGD_multistg}
    \hat{U}_{t+1}= \min_{\chi_t > 0}\left(\min_{\nu_t>0}\left(\min_{\zeta_t>0}\left(\min_{\upsilon_t>0}\,\,\left( \min_{\tau_t\ge 0}\left(\min_{\phi_t h(\tau_t) \le \rho_t^2}\,\, \left( \rho_t^{2} \hat{U}_t +  F(\chi_t,\nu_t,\zeta_t,\upsilon_t)\right) \right)\right)\right) \right) \right)
\end{align}
where $F(\chi_t,\nu_t,\zeta_t,\upsilon_t)$ is defined as
\[F(\chi_t,\nu_t,\zeta_t,\upsilon_t):= \psi_t(\chi_t)^{-1}\left( 1+\nu_t\right)\left(  \frac{c_t^2}{\nu_t} + \alpha_t^2\left( 1+\zeta_t\right)\left( \frac{\sigma_t^2}{\zeta_t } +4 \left( 1+\frac{1}{\upsilon_t} \right) L_g^2 \right)\right).
\]
Here we use the notation $\psi_t(\chi_t):=\left( 1+\frac{1}{\chi_t}\left(\frac{\chi_t-1}{2}\right)^2 \right)^{-1}$.
Then we can just solve \eqref{eq:IPOGD_multistg} in the following layer-by-layer manner: 
\begin{equation*}
\begin{split}
     \hat{U}_{t+1}&=\min_{\chi_t > 0}\left(\min_{\nu_t>0}\left(\min_{\zeta_t>0}\left(\min_{\upsilon_t>0}\,\,\left( \min_{\tau_t\ge 0}\,\, \left( \phi_t h(\tau_t) \hat{U}_t +  F(\chi_t,\nu_t,\zeta_t,\upsilon_t)\right) \right)\right)\right) \right)\\
     &=\min_{\chi_t > 0}\left(\min_{\nu_t>0}\left(\min_{\zeta_t>0}\left(\min_{\upsilon_t>0}\left( \phi_t \left( \min_{\tau_t\ge 0} h(\tau_t) \right) \hat{U}_t +  F(\chi_t,\nu_t,\zeta_t,\upsilon_t)\right) \right)\right)\right)\\
     &=\min_{\chi_t > 0}\left(\min_{\nu_t>0}\left(\min_{\zeta_t>0}\left(\min_{\upsilon_t>0}\left( \phi_t \mu_t^2 \hat{U}_t +  F(\chi_t,\nu_t,\zeta_t,\upsilon_t)\right) \right)\right)\right)\\
    &=\min_{\nu_t>0} \left( \left( 1+\nu_t\right) \left(\min_{\zeta_t>0} \left( \frac{c_t^2}{\nu_t}  +\left( 1+\zeta_t\right) \left(
    \frac{\alpha_t^2 \sigma_t^2}{\zeta_t}+ \left(\mu_t \sqrt{\hat{U}_t}+2 \alpha_t L_g \right)^2 \right) \right) \right) \right)\\
    &=\min_{\nu_t>0} \left( \left( 1+\nu_t\right) \left(  \frac{c_t^2}{\nu_t}  +\left( \mu_t \sqrt{\hat{U}_t} +\alpha_t \left( \sigma_t +2 L_g \right) \right)^2 \right) \right)\\
     &=\left( \mu_t \sqrt{\hat{U}_t}+ \alpha_t \sigma_t + c_t + 2 \alpha_t L_g\right)^2.
\end{split}    
\end{equation*}
Our proof is complete.
\end{proof}

\begin{remark}
The recursive bound in \eqref{Eq:U_inexact_proxGrad} can be further simplified to $\hat{U}_{t+1}=\left( \mu \sqrt{\hat{U}_t}+ \alpha \sigma + c + 2 \alpha L_g\right)^2$ by taking $L_t=L$, $m_t=m$, $\alpha_t=\alpha$, $c_t=c$ and $\sigma_t=\sigma$ for all $t$. The simplified iterative bound leads us to the following asymptotic bound: 
\[
\sqrt{\hat{U}_t}= \mu^t \left( \sqrt{\hat{U}_0} - \frac{\alpha\sigma+c +2 \alpha L_g}{1-\mu}\right)+ \frac{\alpha\sigma+c +2 \alpha L_g}{1-\mu}
\]
Clearly, $\sqrt{\hat{U}_t}$ converges linearly to its steady state value at the rate quantified by $\mu$.
\end{remark}

\section{Dynamic Regret Analysis}
\label{sec:dynamicregret}
There is a routine which can be used to convert 
the tracking error bounds derived in the previous sections into new upper bounds for \emph{dynamic regret}. Specifically, the following lemma is useful.
\begin{lemma}\label{lem:regret_lemma}
Suppose $f_t$ is L-smooth, $\nabla f_t(x_t^*)=0$, and $\norm{x_t-x_t^*}^2\le \hat{U}_t$ $\forall t$. If the tracking error bound $\hat{U}_t$~satisfies the following inequality with some $0\le \gamma<1$ and $u_t\ge 0$:
\begin{equation}\label{eq:itr_mdl}
    \sqrt{\hat{U}_{t+1}} \leq  \gamma \sqrt{\hat{U}_t} + u_t,
\end{equation}
 then the dynamic regret of the online optimization method can be bounded as
\begin{align}\label{eq:regret_gen}
     \sum_{t=0}^T (f_t(x_t)-f_t(x_t^*)) \leq \frac{ L }{ \left(1- \gamma\right)^2} \hat{U}_0 + \frac{L }{\left(1- \gamma \right)^2} \left(\sum\limits_{t=0}^{T-1} u_t\right)^2.
\end{align}
\end{lemma}
\begin{proof}
Since $f_t$ is $L$-smooth and $\nabla f_t(x_t^*)=0$, we have
    $f_t(x_t) \leq f_t(x_t^*)  +\frac{L}{2}\norm{x_t -x_t^*}^2$. Therefore, 
 the dynamic regret satisfies $\sum_{t=0}^T (f_t(x_t)-f_t(x_t^*))\le \frac{L}{2}\sum_{t=0}^T\norm{x_t-x_t^*}^2\le \frac{L}{2}\sum_{t=0}^T \hat{U}_t\le \frac{L}{2}\left(\sum_{t=0}^T \sqrt{\hat{U}_t}\right)^2$.
Based on \eqref{eq:itr_mdl}, we have $\sqrt{\hat{U}_t}\le \gamma^t \sqrt{\hat{U}_0}+\gamma^{t-1} u_0+\gamma^{t-2} u_1+\cdots+u_{t-1}$. This leads to the following bound
\begin{align*}
    \sum_{t=0}^T \sqrt{\hat{U}_t}\le \sum_{t=0}^T \gamma^t \sqrt{\hat{U}_0}+\frac{1}{1-\gamma}\sum_{t=0}^{T-1} u_t\le \frac{1}{1-\gamma}\sqrt{\hat{U}_0}+\frac{1}{1-\gamma}\sum_{t=0}^{T-1} u_t.
\end{align*}
Therefore, we have  $\left(\sum_{t=0}^T \sqrt{\hat{U}_t}\right)^2\le \left(2\hat{U}_0+2(\sum_{t=0}^{T-1} u_t)^2\right)/(1-\gamma)^2$. This directly leads to the desired conclusion.
\end{proof}
Based on the above result, it is almost a trivial task to convert the tracking error bounds $\hat{U}_t$ into dynamic regret bounds. We briefly summarize these result next. 

\vspace{0.1in}

\noindent
$\bullet\,\,$\textbf{Inexact OGD with absolute error.}  If   $\alpha_t=\alpha$ and $\mu_t=\mu$ for all $t$, we have $\sqrt{\hat{U}_{t+1}}\le \mu\sqrt{\hat{U}_t}+\sigma_t+\alpha c_t$. Based on Lemma~\ref{lem:regret_lemma}, the following bound holds 
    \begin{align*}
         \sum_{t=0}^T (f_t(x_t)-f_t(x_t^*)) \leq \frac{ L }{ \left(1- \mu\right)^2} \hat{U}_0 + \frac{2L }{\left(1- \mu \right)^2} \left(\sum\limits_{t=0}^{T-1} \sigma_t \right)^2+\frac{2L\alpha^2 }{\left(1- \mu \right)^2} \left(\sum\limits_{t=0}^{T-1} c_t \right)^2.
    \end{align*}

\noindent
$\bullet\,\,$\textbf{Inexact OGD with relative error.} If $\alpha_t=\alpha$ and $\hat{\rho}_t=\hat{\rho}$ for all $t$, we have $\sqrt{\hat{U}_{t+1}}\le \hat{\rho}\sqrt{\hat{U}_t}+\sigma_t$. Based on Lemma~\ref{lem:regret_lemma}, the following bound holds 
    \begin{align*}
         \sum_{t=0}^T (f_t(x_t)-f_t(x_t^*)) \leq \frac{ L }{ \left(1- \hat{\rho}\right)^2} \hat{U}_0 + \frac{L }{\left(1- \hat{\rho} \right)^2} \left(\sum\limits_{t=0}^{T-1} \sigma_t \right)^2.
    \end{align*}

\noindent
$\bullet\,\,$\textbf{Stochastic OGD with additive noise.}
We can use the stochastic variant of Lemma~\ref{lem:regret_lemma} to obtain the stochastic dynamic regret bound. Specifically, we have $\mathbb{E}\norm{x_t-x_t^*}^2\le \hat{U}_t$ and $\mathbb{E}\left( f_t(x_t)-f_t(x_t^*) \right)\le \frac{L}{2}\mathbb{E}\norm{x_t-x_t^*}^2$. From Theorem \ref{th:stoc_OGD_param}, we have the bound $\sqrt{\hat{U}_{t+1}}\le \mu \sqrt{\hat{U}_t}+\sigma_t+\alpha c_t$.
Then we can slightly modify the proof of Lemma~\ref{lem:regret_lemma} to show
 \begin{align*}
         \sum_{t=0}^T \mathbb{E}(f_t(x_t)-f_t(x_t^*)) \leq \frac{ L }{ \left(1- \mu\right)^2} \hat{U}_0 + \frac{2L }{\left(1- \mu \right)^2} \left(\sum\limits_{t=0}^{T-1} \sigma_t \right)^2+\frac{2L\alpha^2 }{\left(1- \mu \right)^2} \left(\sum\limits_{t=0}^{T-1} c_t \right)^2.
    \end{align*}

\noindent
$\bullet\,\,$\textbf{Stochastic OGD in the finite-sum setting.}
From Theorem \ref{th:SGD_fin_sum_param}, we have the tracking error bound $\sqrt{\hat{U}_{t+1}}\le \sqrt{\hat{\rho}}\hat{U}_t+\sqrt{2}\alpha G_t+\sigma_t$. Similarly, we can immediately obtain the stochastic dynamic regret bound:
\begin{align*}
         \sum_{t=0}^T \mathbb{E}(f_t(x_t)-f_t(x_t^*)) \leq \frac{ L \hat{U}_0 }{ \left(1- \sqrt{\hat{\rho}}\right)^2} + \frac{2L }{\left(1- \sqrt{\hat{\rho}} \right)^2} \left(\sum\limits_{t=0}^{T-1} \sigma_t \right)^2+\frac{4L\alpha^2 }{\left(1- \sqrt{\hat{\rho}} \right)^2} \left(\sum\limits_{t=0}^{T-1} G_t \right)^2
    \end{align*}
\noindent
$\bullet\,\,$\textbf{IP-OGD.} In this case, we need to slightly modify the analysis. Specially, we are interested in deriving an upper bound for $\sum_{t=0}^T\mathbb{E}(h_t(x_t)-h_t(x_t^*))$. We have
\begin{align*}
  \sum_{t=0}^T\mathbb{E}(h_t(x_t)-h_t(x_t^*))\le \sum_{t=0}^T\mathbb{E}(f_t(x_t)-f_t(x_t^*))+\sum_{t=0}^T\mathbb{E}(g_t(x_t)-g_t(x_t^*))
\end{align*}
Notice that now we have $\nabla f_t(x_t^*)\neq 0$. Since $x_t^*$ is the global minimum of $f_t+g_t$, we must have $-\nabla f_t(x_t^*)\in \partial g_t(x_t^*)$. Since $g_t$ is $L_g$-Lipschitz, we have $\norm{\nabla f_t(x_t^*)}\le L_g$. Now we have
\begin{align*}
  \sum_{t=0}^T\mathbb{E}(h_t(x_t)-h_t(x_t^*))\le\frac{L}{2}\sum_{t=0}^T\hat{U}_t+2L_g\sum_{t=0}^T \mathbb{E}\norm{x_t-x_t^*}\le \frac{L}{2}\sum_{t=0}^T\hat{U}_t+2L_g\sum_{t=0}^T\sqrt{\hat{U}_t}.
\end{align*}
Suppose \eqref{eq:itr_mdl} holds. We can easily modify the proof of Lemma~\ref{lem:regret_lemma} to show
\begin{align*}
  \sum_{t=0}^T\mathbb{E}(h_t(x_t)-h_t(x_t^*))\le \frac{ L \hat{U}_0}{ \left(1- \gamma\right)^2} + \frac{2L_g\sqrt{\hat{U}_0}}{1-\gamma} + \frac{L }{\left(1- \gamma \right)^2} \left(\sum\limits_{t=0}^{T-1} u_t\right)^2+\frac{2L_g}{1-\gamma}\sum_{t=0}^{T-1} u_t.
\end{align*}
For IP-OGD, if $\alpha_t=\alpha$ and $\mu_t=\mu$ for all $t$, we have $\gamma=\mu$ and $u_t=\alpha \sigma_t+c_t+2\alpha L_t$. Therefore, we have the above stochastic dynamic regret bound with $\gamma=\mu$ and $u_t=\alpha \sigma_t+c_t+2\alpha L_t$.

\section{Conclusion}
In this paper, we develop analytical solutions for sequential SDPs to yield upper bounds for the tracking error and dynamic regret of a large family of inexact online optimization methods. Our analysis provides a unified treatment of inexact online optimization methods in both the determinisitc and stochastic settings, and addresses the oracle inexactness in a versatile manner.  In the future, it will be interesting to investigate the analysis of inexact online mirror descent methods, which requires the use of Bregman divergence as the distance metric for the tracking error. 
 
\appendix

\section{Proof of Theorem \ref{th:optimal_param_VI}}
\label{app:proof_VI}
\begin{proof}
Based on Proposition \ref{Prop:Greedy_app}, we need to solve the following problem:
\begin{equation}\label{Eq:VI_Opt_stg1}
\hat{U}_{t+1}= \minimize_{(\rho_t,\lambda_t^{(1)},\lambda_t^{(2)},\lambda_t^{(3)},\lambda_t^{(4)}  )\in\mathcal{T}_t} \rho_t^2 \hat{U}_t + \lambda_t^{(2)} \sigma_t^2 +\lambda_t^{(3)} c_t^2.
\end{equation}
Set $\lambda_t^{(2)} =1+\frac{1}{\nu_t}$ and $\lambda_t^{(3)} =\alpha_t^2(1+\nu_t)\left( 1+\frac{1}{\zeta_t} \right)$. Then we can apply the Schur complement twice to show that \ref{eq:LMI_VI} is equivalent to $\zeta_t>0,\,\,\nu_t>0$ and 
\begin{equation*}
\bmat{\psi(\zeta_t,\nu_t)-\rho_t^2+ \lambda_t^{(4)}L_t^2 -2 \lambda_t^{(1)}  m_t   & \lambda_t^{(1)}-\alpha_t\psi(\zeta_t,\nu_t) \\
 \lambda_t^{(1)}-\alpha_t\psi(\zeta_t,\nu_t)   & \alpha_t^2 \psi(\zeta_t,\nu_t) (1+\frac{\delta_t^2}{\zeta_t})- \lambda_t^{(4)} }\preceq 0
\end{equation*}
where $\psi(\zeta_t,\nu_t):=(1+\zeta_t)(1+\nu_t)$. Set $\tau_t= \lambda_t^{(4)}-\alpha_t^2\psi(\zeta_t,\nu_t)(1+\frac{\delta_t^2}{\zeta_t})$. We have $\lambda_t^{(4)}= \tau_t +\alpha_t^2\psi(\zeta_t,\nu_t)(1+\frac{\delta_t^2}{\zeta_t})$. Applying the Schur complement, we can show that the above condition holds if and only if $\nu_t >0 $, $\zeta_t >0 $, $\tau_t \geq 0$, and $h(\nu_t,\zeta_t,\tau_t,\lambda_t^{(1)}) \leq \rho_t^2$, where $h(\nu_t,\zeta_t,\tau_t,\lambda_t^{(1)}):= \psi_t(\nu_t,\zeta_t) \left( 1+\alpha_t^2 L_t^2(1+\frac{\delta_t^2}{\zeta_t}) -\frac{2 \alpha_t \lambda_t^{(1)} }{\tau_t}  \right) + \tau_t L_t^2 -2 \lambda_t^{(1)} m_t +
     \frac{\alpha_t^2\psi_t(\nu_t,\zeta_t)^2 + (\lambda_t^{(1)})^2}{\tau_t}$.
For any $\nu_t> 0$, $\zeta_t>0$ and $\tau_t\ge 0$,
the optimal choice of $\lambda_t^{(1)}$ that minimizes $h$ is given by $\lambda_t^{(1)}=\tau_t m_t+\alpha_t\psi_t\ge 0$.
We have
\begin{align*}
    h(\nu_t,\zeta_t, \tau_t, \lambda_t^{(1)})  & \ge h(\nu_t,\zeta_t,\tau_t, \tau_t m_t+\alpha_t \psi_t) =\psi_t \bar{F}(\zeta_t)+\tau_t(L_t^2-m_t^2) \geq 0,
\end{align*}
where $\bar{F}(\zeta_t):=1-2\alpha_t m_t+\alpha^2 L_t^2\left(1+\frac{\delta_t^2}{\zeta_t}\right)$. Then we can denote $F(\nu_t,\zeta_t):= \left( 1+\frac{1}{\nu_t}\right) \sigma_t^2  + \alpha_t^2(1+\nu_t)\left( 1+\frac{1}{\zeta_t} \right) c_t^2$ and solve \eqref{Eq:VI_Opt_stg1} as follows
\begin{align*}
    \hat{U}_{t+1}&=\min_{\nu_t>0}\,\,\left( \min_{\zeta_t>0}\,\,\left( \min_{\tau_t\ge 0} \,\,\left( \min_{\lambda_t^{(1)}\ge 0} \,\,\left( \min_{\rho_t^2 \ge h\left(\nu_t, \zeta_t,\tau_t,\lambda_t^{(1)} \right) }  \left(\rho_t^2 \hat{U}_t + F(\nu_t,\zeta_t) \right) \right)\right) \right) \right) \\
    &=\min_{\nu_t>0}\,\,\left( \min_{\zeta_t>0}\,\,\left(  \min_{\tau_t\ge 0} \,\,\left( \min_{\lambda_t^{(1)}\ge 0} \,\,\left(  h(\nu_t, \zeta_t,\tau_t,\lambda_t^{(1)} )\hat{U}_t + F(\nu_t,\zeta_t) \right) \right)\right) \right) \\
     &=\min_{\nu_t>0}\,\,\left(\min_{\zeta_t>0}\,\,\left(  \min_{\tau_t\ge 0} \,\,\left(  \psi_t(\nu_t,\zeta_t) \bar{F}(\zeta_t )\hat{U}_t+\tau_t(L_t^2-m_t^2)\hat{U}_t + F(\nu_t,\zeta_t)  \right) \right) \right) \\
     &=\min_{\nu_t>0}\,\,\left( \min_{\zeta_t>0}\,\,\left( \left( 1+\nu_t\right) \left( \left( 1+\zeta_t\right)\left( \bar{F}(\zeta_t) \hat{U}_t +\frac{\alpha_t^2 c_t^2}{\zeta_t}\right) +\frac{\sigma_t^2}{\nu_t}  \right)\right) \right) \\
     &=\min_{\nu_t>0}\,\,\left( \left( 1+\nu_t\right)  \left( \left( \sqrt{\left(1-2\alpha_t m_t+\alpha_t^2L_t^2 \right) \hat{U}_t } +\alpha_t \sqrt{c_t^2+\delta_t^2 L_t^2 \hat{U}_t}\right)^2 +\frac{\sigma_t^2 }{\nu_t} \right)  \right) \\
     &= \left( \sqrt{(1-2 m_t \alpha_t + \alpha_t^2 L_t^2) \hat{U}_t} +\alpha_t \sqrt{c_t^2+\delta_t^2 L_t^2 \hat{U}_t}+\sigma_t \right)^2.
\end{align*}
\end{proof}

\section{Explanations of Remark \ref{remark2}}
The sequential optimization problem in \cite[Proposition 1]{hu2021analysis} is equivalent to the following sequential SDP:
\begin{equation}\label{eq:Remark2_optimzation}
    \hat{U}_{t+1}= \minimize_{(\rho_t,\lambda_t^{(1)},\lambda_t^{(3)},\lambda_t^{(4)}  )\in\mathcal{T}_t} \rho^2_t \hat{U}_t +2 \lambda_t^{(4)} G^2 +\lambda_t^{(3)} c^2,
\end{equation} 
where $\mathcal{T}_t \subseteq \R_+^4$ is the set of tuples $( \rho_t, \lambda_t^{(1)}, \lambda_t^{(3)}, \lambda_t^{(4)})$ that satisfy \eqref{eq:LMI_VI} with $\lambda_t^{(2)}=0$.
It is worth mentioning that the above optimization is similar to the sequential SDP in \eqref{Eq:VI_Opt_stg1} with $\sigma_t \equiv 0$ and a slightly modified cost function.
 Next, we can solve \eqref{eq:Remark2_optimzation} analytically to show that the recursive bound in \cite[Remark 3]{hu2021analysis} is indeed the exact solution for the sequential SDP in \cite[Proposition 1]{hu2021analysis}. Setting $\lambda_t^{(3)} =\alpha_t^2\left( 1+\frac{1}{\zeta_t}\right)$ and repeating the steps performed in Appendix \ref{app:proof_VI} for the proof of Theorem \ref{th:optimal_param_VI} directly leads to the following bound:

\begin{align*}
    \hat{U}_{t+1}&=\min_{\zeta_t>0}\,\,\left(  \left( 1+\zeta_t\right) \left( \left( 1- 2 \alpha_t m +2 \alpha_t^2 L^2 \right)\hat{U}_t +2 \alpha_t^2 G^2 \left( 1+\frac{\delta^2}{\zeta_t}\right) +\frac{\alpha_t^2 c^2}{\zeta_t}\right)  \right) \\
    &=\left( \alpha_t \sqrt{c^2 +2 \delta^2 G^2 +2 \delta^2 L^2 \hat{U}_t} + \sqrt{  \left(1- 2 \alpha_t m +2 \alpha_t^2 L^2 \right) \hat{U}_t +2 G^2 \alpha_t^2} \right)^2.
\end{align*}

\bibliographystyle{abbrv}

\bibliography{bib_file}

\end{document}